\documentclass[11pt]{amsart}
\usepackage{hyperref}
\usepackage{amsfonts,mathrsfs,bbm,latexsym,rawfonts,amsmath,amssymb,amsthm,epsfig}
\usepackage{fullpage, setspace, color}
\usepackage{indentfirst}

\newtheorem{theorem}{Theorem}[section]
\newtheorem{lemma}[theorem]{Lemma}
\newtheorem{corollary}[theorem]{Corollary}
\newtheorem{property}[theorem]{Property}
\newtheorem{thm}{Theorem}[section]

\newtheorem{remark}[thm]{Remark}
\newtheorem{definition}{Definition}[section]
\numberwithin{equation}{section}

\def\<{\left\langle} \def\>{\right\rangle}
\def\({\left(} \def\){\right)}

\makeatletter 
\renewcommand{\section}{\@startsection{section}{1}{0mm}
	{-\baselineskip}{0.5\baselineskip}{\bf\leftline}}
\makeatother

\title
{Dynamics and spreading speed of a reaction-diffusion system with advection modeling West Nile virus}


\email{chengchengcheng16@mails.ucas.edu.cn}
\email{zhzheng@amt.ac.cn}

\thanks{The authors are supported by the NSF of China(No. 11671382), CAS Key Project of Frontier Sciences(No. QYZDJ-SSW-JSC003), the Key Lab. of Random Complex Structures and Data Sciences CAS and National Center for Mathematics and Interdisplinary Sciences CAS}

\thanks{$^*$ Corresponding author: Zuohuan Zheng}

\begin{document}
	\maketitle
	
	\centerline{ Chengcheng Cheng}
	\medskip
	{\footnotesize
		\centerline{Academy of Mathematics and Systems Science, Chinese Academy of Sciences}
		\centerline{Beijing 100190, China}
		\centerline{University of Chinese Academy of Sciences}
		\centerline{Beijing 100049, China}
	} 
	
	\medskip
	
	\centerline{ Zuohuan Zheng$^*$}
	\medskip
	{\footnotesize
		\centerline{Academy of Mathematics and Systems Science, Chinese Academy of Sciences}
		\centerline{Beijing 100190, China}
		\centerline{University of Chinese Academy of Sciences}
		\centerline{Beijing 100049, China}
		\centerline{College of Mathematics and Statistics, Hainan Normal University}
		\centerline{Haikou, Hainan 571158, China}
	}
	
	\bigskip
	
	
	\begin{abstract}
	This paper aims to explore the temporal-spatial spreading and asymptotic behaviors of West Nile virus (WNv) by a reaction-advection-diffusion system with free boundaries, especially considering the impact of advection term on the extinction and persistence of West Nile virus. We define the spatial-temporal risk index $R^{F}_{0}(t)$ with the advection rate ($\mu$) and the general basic disease reproduction number $R^D_0$ to get the vanishing-spreading dichotomy regimes of West Nile virus. We show that there
	exists a threshold value $\mu^{*}$ of the advection rate, and obtain the threshold results of                                                                                                                                                                                                                                                                                                                                                                                                                                                                                                                                                                                                                                                                                                                                                                                                                                                                                                                                                                                                                                                                                                                                                                                                                                                                                                                                                                                                                                                                                                                                                                                                                                                                                                                                                                                                                                                                                                                                                                                                                                                                                                                                                                                                                                                                                                                                                                                                                                                                                                                                                                                                                                                                                                                                                                                                                                                                                                                                                                                                                                                                                                                                                                                                                                                                                                                                                                                                                                                                                                                                                                                                                                                                                                                                                                                                                                                                                                                                                                                                                                                                                                                                                                                                                                                                                                                                                                                                                                                                                                                                                                                                                                                                                                                                                                                                                                                                                                                                                                                                                                                                                                                                                                                                                                                                                                                                                                                                                                                                                                                                                                                                                                                                                                                                                                                                                                                                                                                                                                                                                                                                                                                                                                                                                                                                                                                                                                                                                                                                                                                                                                                                                                                                                                                                                                                                                                                                                                                                                                                                                                                                                                                                                                                                                                                                                                                                                                                                                                                                                                                                                                                                                                                                                                                                                                                                                                                                                                                                                                                                                                                                                                                                                                                                                                                                                                                                                                                                                                                                                                                                                                                                                                                                                                                                                                                                                                                                                                                                                                                                                                                                                                                                                                                                                                                                                                                                                                                                                                                                                                                                                                                                                                                                                                                                                                                                                                                                                                                                                                                                                                                                                                                                                                                                                                                                                                                                                                                                                                                                                                                                                                                                                                                                                                                                                                                                                                                                                                                                                                                                                                                                                                                                                                                                                                                                                                                                                                                                                                                                                                                                                                                                                                                                                                                                                                                                                                                                                                                                                                                                                                                                                                                                                                                                                                                                                                                                                                                                                                                                                                                                                                                                                                                                                                                                                                                                                                                                                                                                                                                                                                                                                                                                                                                                                                                                                                                                                                                                     $\mu^*$. When the spreading occurs, we investigate the asymptotic dynamical behaviors of the solution in the long run and first give a sharper estimate that the asymptotic spreading speed of the leftward front is less than the rightward front in the case of $0<\mu<\mu^*$. At last, we give some numerical simulations to identify the significant effects of the advection.
	\end{abstract}

	\section{Introduction}\label{s1}
	\noindent
	Infectious disease has become an essential factor which threatens people's life nowadays. West Nile virus (WNv), originated from Africa, is one of the fatal mosquito-borne contagious diseases that has widely spread all over the north America since it broke out in 1999 for the first time \cite{koar2003west,nash2001out}. In 1999--2001, West Nile virus lead to 149 cases of clinical neurologic disease in humans and 11932 deaths in birds in the United States \cite{koma2003exper}. Since 2008, many central European countries were invaded by West Nile virus which resulted in several hundreds of human and animal neuroinvasive cases  \cite{rudo2017west}. In the 77 years since West Nile virus was discovered, this virus has spread to a large region of the earth and is considered one of the most important causative agent of viral encephalitis all over the world \cite{chancey2015global}.
	Therefore, it is urgent to understand how the disease can spread spatially to large region to cause large-scale epidemic and investigate the vanishing-spreading dichotomy regimes of West Nile virus.
	
	In order to investigate the spreading dynamics of WNv, Lewis et al.~\cite{lewis2006traveling} first discussed a reaction-diffusion sysytem of WNv with diffusion terms describing the movement of birds and mosquitoes. Then Maidana and Yang \cite{maidana2009spatial} proposed a reaction-advection-diffusion equation to study the vanishing and spreading of WNv across America. Recently, Li et al. \cite{Li2019vertical} formulated and analyzed a periodic delay differential equation of WNv model with vertical transmission. There are also other studies about WNv, such as, Wonham et al. \cite{wohn2006transmission}, Hartley et al. \cite{hartley2012effect} and references therein.
	
	The free boundary problems associated with the ecological models have attracted considerable research interests in the past, and several results have been applied to lots of  fields, such as \cite{vuik1985numerical, yi2004one, amadori2005singular, chen2000free, chen2003free, tao2005elliptic}. There are also a lot of reaction-diffusion biological models to use free boundary for studying. Du and Lin \cite{du2010spreading} first investigated a diffusive logistic equation with a free boundary in one dimension space.
	Wang \cite{wang2014free} studied the asymptotic behaviors about some free boundary problems for the Lotka-Volterra model of two species. Tarboush et al.\cite{tarboush2017spreading} considered a WNv problem with a coupled system, which described the diffusion of birds by a partial differential equation and the movement of mosquitoes by an ordinary differential equation.  Lin and Zhu \cite{lin2017spatial} put forward a reaction-diffusion system with moving fronts to investigate the spreading dynamics of WNv between mosquitoes and host birds across North America.
	
	Although there have been many works to investigate the  propagation of WNv, advection terms have not received much attention in studying the spreading and vanishing of WNv. However, advection, especially the bird advection, plays an important role in spreading of WNv. For instance, in order to investigate the spreading of WNv in North America, it was observed in \cite{maidana2009spatial} that WNv appeared for the first time in New York city in 1999. In 2004, WNv was detected among birds in California. It has been spread across almost the whole America continent since it broke out in America. Thus, the advection movements of birds and diffusion lead to the biological invasion of WNv from the east to the western coast of the USA.
	Therefore, it is much worthwhile to take into consideration the advection movement in modeling West Nile virus.
	
	Considering the previous preliminaries, in order to more explicitly describe the spreading and vanishing of WNv, we are planning to study a reaction-advection-diffusion epidemic model with free boundaries to describe the spreading and vanishing of WNv on the basis of \cite{maidana2009spatial}. Since the movements of the birds and mosquitoes change with time, we assume that the habitats of the birds and mosquitoes have moving boundaries. And the impact of the advection movement on the asymptotic spreading speeds of the double fronts is mainly investigated when spreading happens. Since the effect of advection rate on the mosquitoes is small enough, so we only consider the impact of advection movement on the birds.
	
	For simplicity, let
	\begin{equation}\label{1.2}
	\begin{aligned}
	&a_{1}:=\frac{\alpha_{1} \beta}{N_{1}}, a_{2}:=\frac{\alpha_{2} \beta}{N_{1}}. \end{aligned}
	\end{equation}
	Here what $\alpha_{1},\alpha_{2},\beta$ represent will be explained later.
	Now we are going to discuss the following simplied reaction-advection-diffusion system with free boundaries of WNv between mosquitoes and birds
	\begin{equation}\label{1.3}
	\left\{\begin{array}{ll}{U_{t}=D_{1} U_{x x}-\mu U_{x}+a_{1}\left(N_{1}-U\right) V-\gamma U,} & {g(t)<x<h(t), \enspace t>0,} \\ {V_{t}=D_{2} V_{x x}+a_{2}\left(N_{2}-V\right) U-d V,} & {g(t)<x<h(t), \enspace t>0,} \\ {U(x, t)=V(x, t)=0,} & {x=h(t) \text { or } x=g(t),\enspace t>0,} \\ {h(0)=h_{0},\enspace h^{\prime}(t)=-\nu U_{x}(h(t), t),} & {t>0,} \\ {g(0)=-h_{0}, \enspace g^{\prime}(t)=-\nu U_{x}(g(t), t),} & {t>0,} \\ {U(x, 0)=U_{0}(x), \quad V(x, 0)=V_{0}(x),} & {-h_{0} \leq x \leq h_{0},}\end{array}\right.
	\end{equation}
	where $U(x,t)$ and $V(x, t)$ represent the spatial infected densities of birds and mosquitoes at location $x$ and time $t$, respectively; $N_1$ and $N_2$ are the total carrying capacities of the birds and mosquitoes ; $D_{1}$ and $D_{2}$ represent the diffusion rates of the birds and mosquitoes and $D_2\ll D_1$; $\mu $ is the advection rate caused by the wind on the birds; $\alpha_{1}$ and $\alpha_{2}$ represent the WNv transmission probabilities per bite to birds and mosquitoes; $\beta$ is the biting rate of mosquitoes to birds; $\gamma$ is the recovery rate of birds from infection; $d$ is the death rate of the mosquitoes. The moving region $(g(t),h(t))$ represents the infected habitat of WNv. Suppose that the free boundaries satisfy Stefan conditions: $$g^{\prime}(t)=-\nu U_{x}(g(t), t)$$ and  $$ h^{\prime}(t)=-\nu U_{x}(h(t), t),$$ where $\nu$ is a positive constant which represents the boundary expanding capacity.
	Furthermore, we assume  that the initial conditions $U_{0}$ and $V_{0}$ satisfy
	\begin{equation}\label{1.4}
	\left\{\begin{array}{ll}{U_{0} \in C^{2}\left[-h_{0}, h_{0}\right],} & {U_{0}(\pm h_{0})=0, \enspace 0<U_{0}(x) \leq N_{1} \text { in }\left(-h_{0}, h_{0}\right),} \\ {V_{0} \in C^{2}\left[-h_{0}, h_{0}\right],} & {V_{0}(\pm h_{0})=0,\enspace 0<V_{0}(x) \leq N_{2} \text { in }\left(-h_{0}, h_{0}\right).}\end{array}\right.
	\end{equation}
	Considering small advection rate and high risk at infinity, in this paper, we make the following hypothesis $$(H) \quad\quad\quad {a_{1}a_{2}N_{1}N_{2}}> d\gamma, |\mu|<\mu^{*},$$
	where $\mu^*:=2\sqrt{D_{1}(\dfrac{a_{1}a_{2}N_{1}N_{2}}{d}-\gamma)},$ which is defined as a threshold value of the advection (see to Section \ref{s3}).
	What is more, as Guo and Wu \cite{guo2012free}, we define $\lim\limits_{t\rightarrow +\infty}\limits\frac{-g(t)}{t}$ and  $\lim\limits_{t\rightarrow +\infty}\limits\frac{h(t)}{t}$ as the leftward and rightward asymptotic spreading speeds, respectively.

	As is discussed before, the main goal of this paper is to explore the propagation of WNv and investigate the effect of the advection movement on the asymptotic behaviors by a more general epidemic system (\ref{1.3}).  In view of the bird avection movement ($\mu$), the diffusion of birds and mosquitoes ($D_1,D_2$)  and moving infected regions ($(g(t),h(t))$), this reaction-advection-diffusion model is more coincident with the laws of WNv than \cite{lewis2006traveling} and \cite{lin2017spatial}. By introducing the spatial-temporal risk index $R^F_{0}(t)$ with respect to advection and time as a threshold condition, the vanishing-spreading dichotomy regimes of WNv are obtained (see to Theorem \ref{t54}). It is worthy of note that spreading speed is an important factor to influence the frontier propagation rate of WNv, while many previous studies can not calculate the spreading speed of the epidemic. Our main result is that we give an estimate that the asymptotic spreading speed of the leftward front is less than the rightward front for $0<\mu<\mu^*$ when the spreading occurs (see to Theorem \ref{t56}).
	
	This paper is organized as follows: in section 2, we prove the existence and uniqueness of the solution for system (\ref{1.3}) by contraction mapping theorem, standard $L^p$ estimate and Sobolev embedding theorem. In section 3, we introduce spatial-temporal risk index $R^F_{0}(t)$ and the general basic reproduction number $R^D_0$. In section 4 and 5, we discuss the vanishing-spreading dichotomy regimes of WNv by applying $R^F_{0}(t)$ and $R^D_0$. In section 6, we mainly give an explicit estimate of the asymptotic spreading speeds about the leftward front and the rightward front compared with the corresponding reaction-diffusion model without advection. In section 7, several numerical simulations are given to illustrate our analytic results. At last, we sum up this paper by a brief discussion.
	\section{Existence and uniqueness}\label{s2}
	\noindent
	In this section, we first give the basic results about the local existence and uniqueness for the problem (\ref{1.3}) with the initial conditions (\ref{1.4}).
	
	\begin{theorem}\label{t21}\quad
		For any given $(U_{0},V_{0})$ satisfying (\ref{1.4}) and any $\alpha \in (0,1)$, there exists $T>0$ such that the system (\ref{1.3}) admits  a unique solution
		\begin{equation}\label{2.1}
		\begin{aligned}
		&(U,V;g,h)\in ({C^{1+\alpha,(1+\alpha)/2}}(\overline{D_{T}}))^{2}\times (C^{1+\alpha/2}([0,T]))^{2}
		\end{aligned}
		\end{equation}
		Further,
		\begin{equation}\label{2.2}
		||U||_{{C^{1+\alpha,(1+\alpha)/2}}(\overline{D_{T}})}+||V||_{{C^{1+\alpha,(1+\alpha)/2}}(\overline{D_{T}})}
		+||g||_{C^{1+\alpha/2}([0,T])}+||h||_{C^{1+\alpha/2}([0,T])}\leq C_{1}.
		\end{equation}
		Where
		\begin{equation}\label{2.3}
		D_{T}=\{(x,t)\in \mathbb{R}^{2}:x\in(g(t),h(t)),t\in(0,T]\},
		\end{equation}
		positive constants $T, C_{1}$ depend only on $ ||U_0||_{C^{2}([-h_0,h_0])}$, $||V_0||_{C^2([-h_0,h_0])}$,  $h_0$ and $ \alpha$.
	\end{theorem}
	\begin{proof}\quad	Now we only provide a simple sketch to prove this theorem. First, straighten the free boundary; then, use the contraction mapping theorem, standard $L^p$ estimate and Sobolev embedding theorem to get the local existence and uniqueness of $(U, V; g, h)$; finally, apply the Schauder estimates to obtain the regularity of the solution. The detailed proof can refer to Theorem 2.1 in \cite{chen2000free}, Lemma 2.1 in \cite{guo2012free} or Theorem 2.1 Wang and Zhao \cite{wang2017free}.
	\end{proof}
	
	In order to prove the boundness of the local solution, we need to use the following Comparison Principle to eatimate $U(x,t)$, $V(x,t)$ and the free boundaries $x=g(t)$, $x=h(t)$. The proof is similar to Lemma 3.5 in \cite{du2010spreading}, so we omit it here.
	\begin{lemma}[Comparison Principle]\label{l22}\quad
		Assume that $T\in(0,+\infty)$, $\overline{h}(t), \overline{g}(t)\in C^{1}([0,T])$, $\overline{U},\overline{V}\in C(\overline{D_{T}^{*}})\bigcap C^{2,1}(D_{T}^{*})$, and
		\begin{equation}\label{2.4}
		\left\{\begin{array}{ll}{\overline{U}_{t}-D_{1} \overline{U}_{x x} \geq -\mu\overline{U}_{x}+a_{1}\left(N_{1}-\overline{U}\right) \overline{V}-\gamma \overline{U},} & {\overline{g}(t)<x<\overline{h}(t), \enspace 0<t<T,} \\ {\overline{V}_{t}-D_{2} \overline{V}_{x x} \geq a_{2}\left(N_{2}-\overline{V}\right) \overline{U}-d \overline{V},} & {\overline{g}(t)<x<\overline{h}(t), \enspace 0<t<T,} \\ {\overline{U}(0, t) \geq U(0, t), \enspace \overline{V}(0, t) \geq V(0, t),} & {0<t<T,} \\ {\overline{U}(x, t)= 0,\overline{V}(x, t)= 0,} & {x=\overline{g}(t)\enspace or\enspace \overline{h}(t), \enspace 0<t<T,} \\ {\overline{h}^{\prime}(t) \geq-\nu \overline{U}_{x}(\overline{h}(t), t),\overline{g}^{\prime}(t) \leq-\nu \overline{U}_{x}(\overline{g}(t), t),} & {\enspace 0<t<T,} \\\overline{U}(x, 0)\geq {U_{0}(x), \enspace\overline{V}(x, 0)\geq V_{0}(x),} & {-h_{0} \leq x \leq h_{0},}\end{array}\right.
		\end{equation}
		then the solution  $(U,V;g,h)$ of (\ref{1.3}) satisfies
		\begin{equation}\label{2.5}
		\begin{aligned}
		&\overline{ U}(x,t)\geq {U}(x,t), \enspace \overline{V}(x,t)\geq {V}(x,t),
		\\ & \overline{h}(t)\geq {h}(t),\enspace {g}(t)\geq \overline{ g}(t), for\enspace g(t)\leq x\leq h(t),t\in(0,T],
		\end{aligned}
		\end{equation}
		where $D_{T}^{*}=\{(x,t)\in\mathbb{R}^{2}: x\in(\overline g(t),\overline{h}(t)),t\in(0,T]\}$.
	\end{lemma}
	
	\begin{remark}\label{r23}\quad
		Assume that $T\in(0,+\infty)$, $\underline{h}(t), \underline{g}(t)\in C^{1}([0,T])$, $\underline{U}, \underline{V}\in C(\overline{D_{T}^{**}})\bigcap
		C^{2,1}(D_{T}^{**})$, if the reverse inequalities of (\ref{2.4}) are satisfied, then
		\begin{equation}\label{2.6}
		\begin{aligned}
		&\underline{U}(x,t)\leq U(x,t),\enspace \underline{V}\leq V(x,t),
		\\&\underline{h}(t)\leq h(t),\enspace  g(t)\leq\underline{g}(t),for\enspace \underline{g}(t)\leq x\leq\underline{h}(t),t\in(0,T],
		\end{aligned}
		\end{equation}
		where $D_{T}^{**}=\{(x,t)\in\mathbb{R}^{2}: x\in(\underline{g}(t),\underline{h}(t)),t\in(0,T]\}$.
	\end{remark}
	
	\begin{remark}\label{r24}\quad
		If $(\overline{U},\overline{V};\overline{g},\overline{h})$ and $(\underline{U},\underline{V};\underline{g},\underline{h})$ satisfy the above conditions, then they are called the upper and lower solution of (\ref{1.3}), respectively.
	\end{remark}	
	
	The following Lemma gives some estimates of $U(x,t)$ and $V(x,t)$.
	\begin{lemma}\label{l25}\quad
		Assume that $T\in(0,+\infty)$. Let $(U,V;g,h)$ be a solution of (\ref{1.3}) for $t\in(0,T]$, then there exist positive constants $C_{2},C_{3}$ independent of $T$ such that
		\begin{equation}\label{2.7}
		\begin{aligned}
		&0<U(x,t)\leq C_{2}, for\enspace g(t)< x< h(t), 0< t\leq T.
		\\&0<V(x,t)\leq C_{3}, for\enspace g(t)< x< h(t), 0< t\leq T.
		\end{aligned}
		\end{equation}
		Indeed, we can take $C_2=N_1,C_3=N_2$.
	\end{lemma}
	\begin{proof}\quad Since $g(t), h(t)$ are fixed, by the strong maximum principle, we can get
		$$U(x,t)>0,V(x,t)>0 \enspace for\enspace(x,t)\enspace \in (g(t),h(t))\times (0,T].$$
		For the problem (\ref{1.3}), consider the following system
		\begin{equation}\label{2.8}
		\left\{\begin{array}{ll}{u^{\prime}=a_{1}(N_{1}-u)v-\gamma u,} & { 0<t\leq T,} \\{v^{\prime}=a_{2}(N_{2}-v)u-d v,} & {0<t\leq T,} \\ {u(0)=\sup\limits_{x\in[-h_{0},h_{0}]}U_{0}(x),}\\{v(0)=\sup\limits_{x\in[-h_{0},h_{0}]}V_{0}(x).} \end{array}\right.
		\end{equation}
		Then $$u(h(t),t)>U(h(t),t)=0, v(h(t),t)>V(h(t),t)=0$$
		for $t\in(0,T].$
		And $$u|_{t=0}\geq U_0(x), v|_{t=0}\geq V_0(x)$$
		for $x\in [-h_{0},h_{0}].$
		Hence, the solution $(u,v)$ of equation (\ref{2.8}) is an upper solution of the system (\ref{1.3}). Apply the upper and lower solutions theorem (Theorem 2.1 in \cite{pao1995reaction}) and Lemma \ref{l22}, we get
		\begin{equation}\label{2.9}
		{u}\geq U(x,t), {v}\geq V(x,t),for\enspace (x,t)\in[g(t),h(t)]\times[0,T].
		\end{equation}
		Since the solution of the ordinary differential equation (\ref{2.8}) satisfies
		\begin{equation}\label{2.10}
		\begin{aligned}
		&{u}\leq \max\limits\{\sup\limits_{x\in[-h_{0},h_{0}]}U_{0}(x),N_{1}\}=N_1,for\enspace (x,t)\in[g(t),h(t)]\times[0,T].
		\\&{v}\leq \max\limits\{\sup\limits_{x\in[-h_{0},h_{0}]}V_{0}(x),N_{2}\}=N_2,for\enspace (x,t)\in[g(t),h(t)]\times[0,T].
		\end{aligned}
		\end{equation}
		Therefore,
		\begin{equation}\label{2.11}
		\begin{aligned}
		&{U}\leq \max\limits\{\sup\limits_{x\in[-h_{0},h_{0}]}U_{0}(x),N_{1}\}=N_1,for\enspace (x,t)\in[g(t),h(t)]\times[0,T].
		\\&{V}\leq \max\limits\{\sup\limits_{x\in[-h_{0},h_{0}]}V_{0}(x),N_{2}\}=N_2,for\enspace (x,t)\in[g(t),h(t)]\times[0,T].
		\end{aligned}
		\end{equation}
		Similarly, considering the following system
		\begin{equation}\label{2.12}
		\left\{\begin{array}{ll}{\underline{u}^{\prime}=a_{1}(N_{1}-\underline{u})\underline{v}-\gamma \underline{u},} & { 0<t\leq T,} \\{\underline{v}^{\prime}=a_{2}(N_{2}-\underline{v})\underline{u}-d\underline{v},} & {0<t\leq T,} \\ {\underline{u}(0) =\inf\limits_{x\in[-h_{0},h_{0}]}U_{0}(x),} & {t>0,} \\{\underline{v}(0) =\inf\limits_{x\in[-h_{0},h_{0}]}V_{0}(x),} & {t>0,} \end{array}\right.
		\end{equation}
		then $(\underline{u},\underline{v})$ is the lower solution of (\ref{1.3}).
		Moreover,
		\begin{equation}\label{2.13}
		0<\underline{u}\leq U(x,t),0<\underline{v}\leq V(x,t), for\enspace (x,t)\in(g(t),h(t))\times(0,T].
		\end{equation}
		Therefore, take
		$$C_{2}=\max\{\sup\limits_{x\in[-h_{0},h_{0}]}U_{0}(x),N_{1}\}=N_1,C_{3}=\max\{\sup\limits_{x\in[-h_{0},h_{0}]}V_{0}(x),N_{2}\}=N_2.$$
		Then $C_{2},C_{3}$ satisfy (\ref{2.7}).
	\end{proof}
	The following lemma gives an estimate about the upper bound and lower bound of the asymptotic spreading speeds for the leftward front and the rightward front, which is similar to Lemma 2.2 in \cite{du2010spreading} or Lemma 3.3 in \cite{tian2017free}.
	\begin{lemma}\label{l26}\quad
		Assume that $T\in(0,+\infty)$. Let $(U,V;g,h)$ be a solution of (\ref{1.3}) for $t\in(0,T]$, then there exists a positive constant $C$ independent of $T$ such that
		$$0<-g^{\prime}(t), h^{\prime}(t)\leq C, for\enspace 0<t\leq T.$$
	\end{lemma}
	
	The following theorem shows that the solution of (\ref{1.3}) can be extend to $[0,\infty)$.
	\begin{theorem}\label{t26}\quad
		For any given $(U_{0},V_{0})$ satisfying the initial conditions, the solution of free boundary problem (\ref{1.3}) exists and is unique for any $t\in[0,\infty)$.
	\end{theorem}
	\begin{proof}\quad
		For fixed $T_0>0$, the uniqueness and local existence of the solution to the problem (\ref{1.3}) can be obtained following from Theorem \ref{t21}. Now we show the existence of the global solution. If $[0,T_{0}]$ is the maximal existence interval of the solution, we will show $T_{0}=+\infty$. Assuming that $T_{0}<+\infty$, for the fixed $\delta\in(0,T_{0})$, by the $L^{p}$ estimates, the Sobolev's embedding theorem and the L\"{o}lder estimates about parabolic equations, there exists $C_{4}=C_{4}(\delta,T,C_1,C_{2},C_{3})>0$ such that
		$$\|U(x,t)\|_{C^{2}(g(t),h(t))},\|V(x,t)\|_{C^{2}(g(t),h(t))} \leq C_{4}, for \enspace t\in[\delta,T_{0}].$$
		Then there exists a small $\varepsilon>0$ such that the solution of (\ref{1.3}) with initial time $T_{0}-\varepsilon/2$ can be extended to the time $T_{0}+\varepsilon/2$ uniquely by Theorem \ref{t21} and Zorn's lemma. This is contradict to the choice of $T_{0}$. Hence, we have completed the proof.
	\end{proof}
	
	\section{Basic reproduction number}\label{s3}
	\noindent
	In this section, we will define and study the basic reproduction number for the system(\ref{1.3}). According to L{\'{o}}pez-G{\'{o}}mez~\cite{caudevilla2008asymptotic},
	there exist a principal eigenvalue and corresponding unique eigenfunction $(\phi,\psi)$ (subject to a multiplicative positive constant) satisfying the following eigenvalue problem
	\begin{equation}\label{3.1}
	\left\{\begin{array}{ll}{-D_{1} \phi_{x x}=-\mu\phi_{x}+\frac{a_{1} N_{1} }{R_{0}^{D}}\psi-\gamma \phi,} & {x\in (-h_{0},h_{0}),} \\ {-D_{2} \psi_{x x}=\frac{a_{2} N_{2} }{R_{0}^{D}}\phi-d \psi,} & {x\in (-h_{0},h_{0}), } \\ {\phi(x)=\psi(x)=0,} & {x=\pm h_{0}.} \end{array}\right.
	\end{equation}
	By applying a similar way as discussed by Diekmann et al.~\cite{diekmann1990definition}, Allen et al. \cite{Allen2008Asymptotic} or Zhao \cite {zhao2003dynamical} to calculate the principal eigenvalue of (\ref{3.1}), we define the unique positive principal eigenvalue $R^D_0$ with Dirichlet boundary condition and the advection rate for problem (\ref{3.1}) as the general basic reproduction number:
	\begin{equation}\label{3.2}
	\begin{aligned}
	&R_{0}^{D}:=R_{0}^{D}((-h_0,h_0),\mu,D_1,D_2)\\&=\sqrt{\dfrac{a_{1}a_{2}N_{1}N_{2}}{[{D_{1}(\dfrac{\pi}{2h_{0}})^2}+\dfrac{\mu^{2}}{4D_{1}}+\gamma][D_{2}(\dfrac{\pi}{2h_{0}})^{2}+d]}}.
	\end{aligned}
	\end{equation}
	
	Applying the similar variational methods from Cantrell and Cosner \cite{cantrell2004spatial} or Lemma 2.3 in \cite {Huang2010Dynamics}, the following result of $R^{D}_{0}$ holds.
	\begin{lemma}\label{l31}\quad
		$1-R^{D}_{0}$ and $\lambda_{0}$ have the same sign, where $\lambda_{0}$ is the principal eigenvalue of the following problem
		\begin{equation}\label{3.3}
		\left\{\begin{array}{ll}{-D_{1} \phi_{x x}=-\mu\phi_{x}+a_{1} N_{1} \psi-\gamma \phi+\lambda_0\phi,} & {x\in (-h_0,h_0),} \\ {-D_{2} \psi_{x x}=a_{2} N_{2} \phi-d \psi+\lambda_0\psi,} & {x\in (-h_0,h_0), } \\ {\phi(x)=\psi(x)=0,} & {x=\pm h_0 .} \end{array}\right.
		\end{equation}
		where $(\phi(x),\psi(x))>0$ for $x\in (-h_0,h_0)$, $\phi^\prime(-h_0)>0,\psi^\prime(-h_0)>0$ and $\phi^\prime(h_0)<0,\psi^\prime(h_0)<0.$	
	\end{lemma}
	
	Since the boundary $(g(t),h(t))$ changes with time, we introduce the spatial-temporal risk index with the advection rate and time as the basic reproduction number in epidemiology
	\begin{equation}\label{3.4}
	\begin{aligned}
	&R^{F}_{0}(t):=R^D_{0}((g(t),h(t)),\mu,D_{1},D_{2})
	\\&=\sqrt{\dfrac{a_{1}a_{2}N_{1}N_{2}}{[{D_{1}(\dfrac{\pi}{h(t)-g(t)})^2}+\dfrac{\mu^{2}}{4D_{1}}+\gamma][D_{2}(\dfrac{\pi}{h(t)-g(t)})^{2}+d]}}.
	\end{aligned}
	\end{equation}
	
	Following from the definition of $R^{F}_{0}(t)$, we can easily get
	\begin{property}\label{p32}\quad The following properties of $R^{F}_{0}(t)$ hold.
		\\(1) $R^{F}_{0}(t)$ is a positive and monotonically decreasing function of $\mu: R^{F}_{0}(t)\rightarrow 0$ as $\mu \rightarrow+\infty;$
		\\(2) If $\mu\neq 0$, then $R^{F}_{0}(t)\rightarrow $  0 as $D_{1}\rightarrow 0$ and $D_{2}\rightarrow 0$ or $D_1 \rightarrow \infty;$
		\\(3) $R^{F}_{0}(t)$ is strictly monotonically increasing function of $t:$ when $\mu\neq 0,$  if $h(t)-g(t)\rightarrow +\infty$ as $t\rightarrow +\infty$, then $R^{F}_{0}(t)\rightarrow R_0(\mu) $ as $t\rightarrow +\infty$,
		where $R_{0}(\mu)=\sqrt{\dfrac{a_{1}a_{2}N_{1}N_{2}}{(\dfrac{\mu^{2}}{4D_{1}}+\gamma)d}}.$
	\end{property}
	
Since the left boundary $x=g(t) $ is monotonically decreasing and the right boundary $x=h(t)$ is monotonically increasing, there exist $g_\infty\in[-\infty,0)$ and $h_\infty\in(0,\infty]$ such that
	$g_\infty=\lim\limits_{t\rightarrow +\infty}g(t)$ and $h_\infty=\lim\limits_{t\rightarrow +\infty}h(t).$
	Moreover, when $\mu=0$, we suppose that the habitat at far distance is in high risk, that is,
	$$R^D_0((-\infty,0),0,D_1,D_2)>1 \enspace and \enspace  R^D_0((0,\infty),0,D_1,D_2)>1,$$ equivalently,
	${a_{1}a_{2}N_{1}N_{2}}> d\gamma.$
	
	In view of the above properties of $R^F_0(t)$, there exists a threshold value $$\mu^{*}:=2\sqrt{D_{1}(\dfrac{a_{1}a_{2}N_{1}N_{2}}{d}-\gamma)}.$$ If $|\mu|<\mu^{*}$,  then there is a $t_0\geq0$ such that $$R^{F}_{0}(t_0)=R^D_{0}((g(t_0),h(t_0)),
	\mu,D_1,D_2)\geq1$$ and $$
	R^F_0(\infty):=R^D_{0}((g_\infty,h_\infty),
	\mu,D_1,D_2)>1$$ under the assumption of $h_\infty-g_\infty=\infty$;
	if $|\mu|>\mu^{*}$, then $$R^{F}_{0}(t)=R^D_{0}((g(t),h(t)),\mu,D_{1},D_{2})< 1$$ for any $t\geq 0$.

    Considering the above arguments about high-risk habitat at far distance and small advection, we make the assumption of $(H)$.

	\section{The vanishing regime of WNv}\label{s4}
	\noindent
	Now we will introduce the definitions of vanishing and spreading from \cite{ge2015sis}.
	\begin{definition}\quad
		The disease is \textbf{ vanishing} if $h_{\infty}-g_{\infty}<\infty$ and $$\lim\limits_{t\rightarrow +\infty}(||U(\cdot,t)||_{C(g(t),h(t))}+||V(\cdot,t)||_{C(g(t),h(t))})=0;$$
		The disease is \textbf{ spreading} if $h_{\infty}-g_{\infty}=\infty$ and $$\lim\limits_{t\rightarrow +\infty}\sup\limits (||U(\cdot,t)||_{C(g(t),h(t))}+ ||V(\cdot,t)||_{C(g(t),h(t))})>0.$$
	\end{definition}
	The following theorem gives the relationship about the  spreading boundaries with the densities of birds and mosquitoes. We show that if the infected region of WNv is bounded, then the densities of the birds and the mosquitoes will decay to 0 and the asymptotic spreading speeds of the double boundaries will decay to 0, that is, the diseases will be extinct, which is coincident with the biological reality.
	\begin{theorem}\label{t42}\quad
		Let $(U,V;g,h)$ be a solution of (\ref{1.3}), if $ h_{\infty}-g_{\infty}<\infty$, then
		\begin{equation}\label{4.1}
		\lim\limits_{t\rightarrow +\infty}||U(\cdot,t)||_{C(g(t),h(t))}=\lim\limits_{t\rightarrow +\infty}||V(\cdot,t)||_{C(g(t),h(t))}=0
		\end{equation}
		and
		\begin{equation}\label{4.2}
		\lim\limits_{t\rightarrow +\infty}g^{\prime}(t)=\lim\limits_{t\rightarrow +\infty}h^{\prime}(t)=0.	
		\end{equation}
	\end{theorem}
	\begin{proof}\quad We will prove this theorem by two steps. \\
		\textbf{Step 1.} we will show
		\begin{equation}\label{4.3}
		||U||_{{C^{1+\alpha,(1+\alpha)/2}}((g(t),h(t))\times[1,\infty))}+||V||_{{C^{1+\alpha,(1+\alpha)/2}}((g(t),h(t))\times[1,\infty))}\leq \tilde{C},
		\end{equation}
		and
		\begin{equation}\label{4.4}
		||g||_{C^{1+\alpha,(1+\alpha)/2}([1,\infty))}+||h||_{C^{1+\alpha,(1+\alpha)/2}([1,\infty))}\leq \tilde{C}
		\end{equation}
		for any $\alpha\in(0,1)$, where $\tilde{C}=\tilde{C}(\alpha,h_{0},||U_{0}||_{C^{2}([-h_{0},h_{0}])},||V_{0}||_{C^{2}([-h_{0},h_{0}])},g_{\infty},h_{\infty})>0$.
		
		Now we straighthen the free boundaries by making a transformation motivated by Wang\cite{wang2014some,wang2016diffusive}
		\begin{equation}\label{4.5}
		y=\frac{2x}{h(t)-g(t)}-\frac{h(t)+g(t)}{h(t)-g(t)}.
		\end{equation}
		then the boundary $x=g(t)$ changes into $y=-1$ and $x=h(t)$ changes into $y=1$.
		A straightforward calculation gives
		\begin{equation}\label{4.6}
		\begin{aligned}
		&\frac{\partial y}{\partial x}=\frac{2}{h(t)-g(t)}:=\sqrt{A(g(t),h(t),y)}, \frac{\partial^2 y}{\partial x^2}=0,
		\\&\frac{\partial y}{\partial t}=-\frac{y({h^\prime(t)}-{g^\prime(t)})+({h^\prime(t)}+{g^\prime(t)})}{h(t)-g(t)}
		\\&\enspace\enspace:=B(g(t),{g^\prime(t),h(t),{h^\prime(t)},y)}.
		\end{aligned}
		\end{equation}
		Let $W(y,t)=U(x,t), Z(y,t)=V(x,t)$, then $W(y,t)$ satisfies
		\begin{equation}\label{4.7}
		\left\{\begin{array}{ll}{W_{t}-D_{1}AW_{yy}+(\mu \sqrt{A}+B)W_{y}=f(W,Z),} & {y\in (-1,1),t>0,} \\ {W(\pm 1,t)=0 ,} & {t>0, } \\ {W(y,0)=U_{0}( h_0 y),} & {y\in(-1,1).} \end{array}\right.
		\end{equation}
		For any integer $n\geq 0,$ define
		$$W^{n}(y,t)=W(y,t+n), Z^{n}(y,t)=Z(y,t+n),$$
		then (\ref{4.7}) becomes
		\begin{equation}\label{4.8}
		\left\{\begin{array}{ll}{W^{n}_{t}-D_{1}A^{n} W^{n}_{yy}+(\mu \sqrt{A^{n}}+B^{n})W^{n}_{y}=f(W^n,Z^n),} & {y\in [-1,1],t\in(0,3],} \\ {W^n(\pm 1,t)=0 ,} & {t\in (0,3], } \\ {W^n(y,0)=U(\frac{y(h(n)-g(n))+h(n)+g(n)}{2},n),} & {y\in[-1,1],} \end{array}\right.
		\end{equation}
		where $A^n=A(t+n), B^n=B(t+n)$. We can see that $W^n,Z^n,A^n $ and $B^n$ are uniformly bounded on $n$ according  to Theorem \ref{t21} and Lemma \ref{l25}.
		Moreover,
		\begin{equation}\label{4.9}
		\max _{0 \leq t_{1}<t_{2} \leq 3,\left|t_{1}-t_{2}\right| \leq \tau}\left|A^{n}\left(t_{1}\right)-A^{n}\left(t_{2}\right)\right| \leq \frac{8\left(h^{n}(t)-g^{n}(t)\right)^{\prime}}{\left(h^{n}(t)-g^{n}(t)\right)^{3}} \leq \frac{2 C_{1} \tau}{h_{0}^{3}} \rightarrow 0 , as \enspace \tau \rightarrow 0.
		\end{equation}
		where $h^{n}(t)=h(t+n)$ and $g^{n}(t)=g(t+n)$. When $ h_{\infty}-g_{\infty}<\infty,$ we obtain $A^{n} \geq \frac{4}{\left(h_{\infty}-g_{\infty}\right)^{2}}$ for any $n\geq 0$ and $t\in (0,3]$.
		
		Taking $p\gg1$, by applying the interior $L^{P}$ eatimate, there exists a postive constant $\overline{C}$ independent of $n$ such that $||U^n||_{W^{2,1}_{p}([-1,1]\times [1,3])}\leq \overline{C}$ for any $n>0$.
		Therefore,  by Sobolev's embedding theorem, $$||U^n||_{C^{1+\alpha,(1+\alpha)/2}([-1,1]\times [1,3])} \leq \overline{C}.$$ Moreover, $||U||_{C^{1+\alpha,(1+\alpha)/2}([-1,1]\times [n+1,n+3])}\leq \overline{C}$. Similarly, we can get $$||V||_{C^{1+\alpha,(1+\alpha)/2}([-1,1]\times [n+1,n+3])}\leq \overline{C_1}$$ for some $\overline{C_1}>0.$
		
		Since
		$$g^{\prime}(t)=-\nu U_{x}(g(t),t),\enspace U_{x}(g(t),t)=\frac{2}{h(t)-g(t)}W_{y}(-1,t),$$
		$$h^{\prime}(t)=-\nu U_{x}(h(t),t),\enspace U_{x}(h(t),t)=\frac{2}{h(t)-g(t)}W_{y}(1,t).$$
		in $[-1,1]\times [n+1,n+3],$ and $g^{\prime}(t)$ and $h^{\prime}(t)$ are bounded, then
		\begin{equation}\label{4.10}
		||g||_{C^{1+\alpha/2}([n+1,n+3])}+||h||_{C^{1+\alpha/2}([n+1,n+3])}\leq \hat{C},
		\end{equation}
		for some $\hat{C}>0$. Since the rectangles $[-1,1]\times [n+1,n+3]$ overlap and $\overline{C}$, $\overline{C_1}$ and $\hat{C}$ are independent on $n$, take $\tilde{C}=\overline{C}+\overline{C_1}+\hat{C}$, so (\ref{4.3}) and (\ref{4.4}) hold. Moreover,  since $h_\infty-g_\infty<\infty$, we can get $g^{\prime}(t)\rightarrow 0$ and $h^{\prime}(t)\rightarrow 0$ as $t\rightarrow \infty$. Thus (\ref{4.2}) has been proved.
		
		\textbf{Step 2.} we will show
		$	\lim\limits_{t\rightarrow +\infty}||U(\cdot,t)||_{C(g(t),h(t))}=0.$
		
		On the contrary, we assume that
		$$	\lim\limits_{t\rightarrow +\infty}\sup\limits||U(\cdot,t)||_{C(g(t),h(t))}=\theta>0.$$ Then there is a sequence$\{(x_{k},t_{k})\} $ in $(g(t),h(t))\times (0,\infty)$ such that $U(x_{k},t_{k})\geq \frac{\theta}{2}$ for any $ k \in N$ and $t_{k}\rightarrow \infty $ as $k\rightarrow \infty.$ Since
		$$-\infty <g_{\infty}<g_{t}<x_{k}<h_{t}<h_{\infty}<\infty,$$
		then there exists a subsequence $\{x_{k_{n}}\}$ of $\{x_{k}\}$ such that $x_{k_{n}} \rightarrow x_{0}\in(g_{\infty},h_{\infty})$ as $n\rightarrow \infty.$ Without loss of generality, we assume that $x_{k}\rightarrow x_{0}$ as $k\rightarrow \infty.$
		
		Let
		$$U_{k}(x,t)=U(x,t_{k}+t), V_{k}(x,t)=V(x,t_{k}+t),$$
		for $(x,t)\in [g(t_{k}+t),h(t_{k}+t)]\times (-t_{k},\infty).$
		In view of (\ref{2.2}) and (\ref{4.3}), there exists a subsequence $\{(U_{k_{n}},V_{k_{n}})\}$ of $\{(U_{k},V_{k})\}$ such that
		$(U_{k_{n}},V_{k_{n}})\rightarrow (\tilde{U},\tilde{V}) \enspace as\enspace  n\rightarrow \infty.$ And $(\tilde{U},\tilde{V})$ satisfies
		\begin{equation}\label{4.11}
		\left\{\begin{array}{ll}{\tilde{U}_{t}=D_{1} \tilde{U}_{x x}-\mu \tilde{U}_{x}+a_{1}\left(N_{1}-\tilde{U}\right) \tilde{V}-\gamma \tilde{U},} & {g_{\infty}<x<h_{\infty}, \enspace t\in (-\infty,\infty),} \\ {\tilde{V}_{t}=D_{2} \tilde{V}_{x x}+a_{2}\left(N_{2}-\tilde{V}\right) \tilde{U}-d \tilde{V},} & {g_{\infty}<x<h_{\infty}, \enspace t \in (-\infty,\infty),} \end{array}\right.
		\end{equation}
		with $\tilde{U}(h_{\infty},t)=0$ for $t\in (-\infty,\infty).$
		Note that $\tilde{U}(x_{0},0)\geq \frac{\theta}{2}$, then  by strong maximum principle, we get $\tilde{U}>0$ in $(g_{\infty},h_{\infty})\times(-\infty,\infty)$.
		Since $$\tilde{U}_{t}-D_{1} \tilde{U}_{x x}+\mu \tilde{U}_{x}+({a_{1}N_{2}+\gamma})\tilde{U}\geq 0,$$
		applying Hopf Lemma at the point $(h_{\infty},0)$, we can get $\tilde{U}_{x}(h_{\infty},0)<0.$
	    It implies that there exists a $\delta_0>0$ such that
	\begin{equation}\label{4.12}
	U_{x}(h(t_{k_{n}}),t_{k_{n}})=(U_{k_{n}})_{x}(h(t_{k_{n}}),0)\leq-\delta_0<0, for \enspace n\gg 1,
	\end{equation}
	so $h^\prime(t_{k_{n}})\geq \nu \delta_0>0$ as $ n$ sufficiently large. Since $h(t)$ is bounded, then $h^{\prime}(t)\rightarrow 0$ as $t\rightarrow \infty$, so $h^\prime(t_{k_{n}})\rightarrow 0$ as $n\rightarrow \infty $, which is a contradiction. Therefore, $	\lim\limits_{t\rightarrow +\infty}||U(\cdot,t)||_{C(g(t),h(t))}=0.$
		
		For any given $\epsilon >0$, there exists a $T>0$ such that $U(x,t)\leq \epsilon$ for $x\in[g(t),h(t)]$ and $t>T$. Therefore, $V_{t}-D_{2} V_{x x}\leq a_{2}N_{2}\epsilon-\gamma V.$
		By Comparison Principle, we get
		$$\lim\limits_{t\rightarrow +\infty}\sup\limits||V(\cdot,t)||_{C(g(t),h(t))}\leq\frac{a_{2}N_{2}}{\gamma}\epsilon.$$
		Since $\epsilon$ is arbitrary, $\lim\limits_{t\rightarrow +\infty}||V(\cdot,t)||_{C(g(t),h(t))}=0$ holds.
	\end{proof}
	
	Moreover, we have the following result.
	\begin{theorem}\label{t43}\quad
		If $ h_{\infty}-g_{\infty}<\infty$, then $R^D_{0}((g_{\infty},h_{\infty}),\mu,D_{1},D_{2})=R^D_{0}((g_{\infty},h_{\infty}),\mu,\gamma,d)\leq 1.$
	\end{theorem}	
	\begin{proof}\quad  We prove this theorem by contradiction. Assume that $R^D_{0}((g_{\infty},h_{\infty}),\mu,D_{1},D_{2})> 1$, then there exists $T\gg 1$ such that $R^D_{0}((g(T),h(T)),\mu,\gamma,d)>1$. For small $\varepsilon>0,$ according to the continuity of $R^D_{0}((g(T),h(T)),\mu,\gamma,d)$ in $\gamma$ and $d$, we get $$R^D_{0}((g(T),h(T)),\mu,\gamma+\varepsilon,d+\varepsilon)>1$$ with $\varepsilon$ dependent on $T$.
		
		Let $(H(x,t),M(x,t))$ be the solution of
		\begin{equation}\label{4.13}
		\left\{\begin{array}{ll}{H_{t}=D_{1} H_{x x}-\mu H_{x}+a_{1}\left(N_{1}-H\right) M-(\gamma+\varepsilon) H,} & {g(T)<x<h(T), \enspace t>T,} \\ {M_{t}=D_{2} M_{x x}+a_{2}\left(N_{2}-M\right) H-(d+\varepsilon) M,} & {g(T)<x<h(T), \enspace t>T,} \\ {H(g(T), t)=H(h(T), t)=0,} & { t>T,}\\ {M(g(T), t)=M(h(T), t)=0,} & { t>T,}  \\ {H(x,T)=U(x,T),} & {g(T)\leq x\leq h(T),} \\ {M(x,T)=V(x,T),} & {g(T)\leq x\leq h(T).} \end{array}\right.
		\end{equation}
		By maximal principle, it follows that $U(x,t)\geq e^{\varepsilon(t-T)}H(x,t)$ and $V(x,t)\geq e^{\varepsilon(t-T)}M(x,t)$ in $[g(T),h(T)]\times [T,\infty).$
		Moreover, in view of $R^D_{0}((g(t),h(t)),\mu,\gamma+\varepsilon,d+\varepsilon)>1$, by using the upper and lower solution meothod with monotone iterations in \cite{pao2012nonlinear} (also see to \cite{tian2018advection}), we can get
		$$	\lim\limits_{t\rightarrow +\infty}H(x,t)=\tilde{H}(x), \lim\limits_{t\rightarrow +\infty}M(x,t)=\tilde{M}(x)$$ uniformly on  $[g(T),h(T)]$, where $(\tilde{H}(x),\tilde{M}(x))$ is the positive steady solution of (\ref{4.13}) and satisfies
		\begin{equation}\label{4.14}
		\left\{\begin{array}{ll}{-D_{1} \tilde{H}^{\prime\prime}+\mu \tilde{H}^{\prime}=a_{1}\left(N_{1}-\tilde{H}\right) \tilde{M}-(\gamma+\varepsilon) \tilde{H},} & {g(T)<x<h(T), } \\ {-D_{2}\tilde{M}^{\prime\prime}=a_{2}\left(N_{2}-\tilde{M}\right) \tilde{H}-(d+\varepsilon) \tilde{M},} & {g(T)<x<h(T),} \\ {\tilde{H}(g(T))=\tilde{H}(h(T))=0,}\\{\tilde{M}(g(T))=\tilde{M}(h(T))=0.}     \end{array}\right.
		\end{equation}
		Therefore, $$\lim\limits_{t\rightarrow +\infty}H(0,t)=\tilde{H}(0), \lim\limits_{t\rightarrow +\infty}M(x,0)=\tilde{M}(0),$$ it imples that
		\begin{equation}\label{4.145}
		\begin{aligned}
		 U(0,t)\geq e^{\varepsilon(t-T)}\tilde{H}(0)>0,  V(0,t)\geq e^{\varepsilon(t-T)}\tilde{M}(0)>0
		 \end{aligned}
		 \end{equation}
		 in $[T,\infty)$. Since $h_\infty-g_\infty<\infty,$ by Theorem \ref{t42}, we get $U(0,t)\rightarrow 0$, $V(0,t)\rightarrow 0$ as $t\rightarrow \infty$, which is contradict to (\ref{4.145}). Hence, the proof is completed.
	    \end{proof}
	
	The following result is an ordinary corollary of the above theorem, which is similar to the argument in \cite{ge2015sis}. It implies that  $h_\infty$ and $g_\infty$ will be finite or infinite simultaneously under the assumption of $(H)$.
	\begin{corollary}\label{t44}\quad
		Assume that (H) holds, if $h_\infty<\infty$ or $-g_\infty<\infty$, then $h_\infty-g_\infty<\infty$. Moreover, we get $R^D_{0}((g_\infty,h_\infty),\mu,D_{1},D_{2})\leq 1.$
	\end{corollary}
	Next we will give some suffient conditions for vanishing of the virus.
	
	\begin{theorem}\label{t45}\quad
		If $R^F_{0}(0)=R^D_{0}((-h_{0},h_{0}),\mu,D_{1},D_{2})< 1,$ then $h_{\infty}-g_{\infty}<\infty,$ and
		\begin{equation}\label{4.15}
		\lim\limits_{t\rightarrow +\infty}||U(\cdot,t)||_{C(g(t),h(t))}=\lim\limits_{t\rightarrow +\infty}||V(\cdot,t)||_{C(g(t),h(t))}=0,
		\end{equation}
		if given $||U_{0}||_{L_{\infty}}$ and $||V_{0}||_{L_{\infty}}$ are so small.	
	\end{theorem}
	\begin{proof}\quad
		Following from the Lemma \ref{l31}, if $R^F_{0}(0)=R^D_{0}((-h_{0},h_{0}),\mu,D_{1},D_{2})< 1$, then there exists $\lambda_{0}>0$ satisfying
		\begin{equation}\label{4.16}
		\left\{\begin{array}{ll}{-D_{1} \phi_{x x}=-\mu\phi_{x}+a_{1} N_{1} \psi-\gamma \phi+\lambda_{0}\psi,} & {x\in (-h_{0},h_{0}),} \\ {-D_{2} \psi_{x x}=a_{2} N_{2} \phi-d \psi+\lambda_{0}\psi,} & {x\in (-h_{0},h_{0}), } \\ {\phi(x)=\psi(x)=0,} & {x=\pm h_{0},} \end{array}\right.
		\end{equation}
		where $(\phi,\psi)>0$ in $(-h_0,h_0)$.
		
		Next, we will prove two claims.\\
		\textbf{Claim 1.} There exists $L\gg 1$ such that
		\begin{equation}\label{4.17}
		x\phi^{\prime}(x)<L\phi(x) , x\psi^{\prime}(x)<L\psi(x)
		\end{equation}
		for $x \in[-h_{0},h_{0}].$
		
		Now we first discuss the case of $\phi(x)$. In fact, since $\phi^{\prime}(-h_{0})>0$ and $\phi^{\prime}(h_{0})<0$,we denote that $x_{1}$ and $x_{2}$ are the first and last critial point from $-h_{0}$ to $h_{0}$, then $\phi^{\prime}(x_{1})=0, \phi^{\prime}(x_{2})=0$ and
		$-h_{0}< x_{1}\leq x_{2}<h_{0}.$
	   Therefore, there exists a $L_1>0$ such that $x\phi^{\prime}(x)<L_1\phi(x)$ holds for $x\in[-h_{0},x_{1}]$ or $x\in[x_{2},h_{0}]$.
		
		Since $\phi(x)>0$, then
		$$ x\phi^{\prime}(x)\leq x||\phi^{\prime}||_{L^{\infty}([x_{1},x_{2}])}\leq L_2 \min\limits_{[x_{1},x_{2}]}\phi(x)\leq L_2\phi(x)$$ for $x\in[x_{1},x_{2}]$,
		where $L_2\geq \dfrac{h_{0}||\phi^{\prime}||_{L^{\infty}([x_{1},x_{2}])}}{\min\limits_{[x_{1},x_{2}]}\phi(x)}.$ Take $L=\max\{L_1,L_2\}$, then $L$ satisfies the requirement of (\ref{4.17}). Similarly, we can take $L\gg 1$, such that $x\psi^{\prime}(x)<L\psi(x)$ for $x\in[-h_0,h_0].$
		
		\textbf{Claim 2.} There exists $\tilde{L}>0$ such that
		\begin{equation}\label{4.19}
		\frac{1}{\tilde{L}}\leq \dfrac{\phi(x)}{\psi(x)}\leq \tilde{L}, \enspace for \enspace x\in[-h_{0},h_{0}].
		\end{equation}
		
		Indeed, since $\phi^{\prime}(h_{0})<0$ and $\psi^{\prime}(h_{0})<0$, then there exists small $\sigma_{1}>0$ such that  $\phi^{\prime}(x)<\frac{\phi^{\prime}(h_{0})}{2}<0$ and $\psi^{\prime}(x)<\frac{\psi^{\prime}(h_{0})}{2}<0$ for any $x\in [h_{0}-\sigma_{1},h_{0}]$. Let $\tilde{L}_{1}=\max\limits_{[h_{0}-\sigma_{1},h_{0}]}\left\{\dfrac{\phi^{\prime}(x)}{\psi^{\prime}(x)}\right\},$ then $\dfrac{\phi^{\prime}(x)}{\psi^{\prime}(x)}\leq \tilde{L}_1,$ by Cauchy mean value theorem, it follows $\dfrac{\phi(x)}{\psi(x)}=\dfrac{\phi(x)-\phi(h_{0})}{\psi(x)-\psi(h_{0})}
		=\dfrac{\phi^{\prime}(\hat{x})}{\psi^{\prime}(\hat{x})}\leq \tilde{L}_1$
		for any $x\in[h_{0}-\sigma_{1},h_{0}]$ and some $\hat{x}\in [h_{0}-\sigma_{1},h_{0}].$ Let $\tilde{L}_2=\max\limits_{[h_{0}-\sigma_{1},h_{0}]}\left\{\dfrac{\psi^{\prime}(x)}{\phi^{\prime}(x)}\right\},$ then $\dfrac{\psi^{\prime}(x)}{\phi^{\prime}(x)}\leq \tilde{L}_2.$ Similarly, we get
		$\dfrac{\psi(x)}{\phi(x)}\leq \tilde{L}_{2}$
		for any $x\in[h_{0}-\sigma_{1},h_{0}].$ Take $\tilde{L}_{3}=\max\limits\left\{\tilde{L}_{1},\tilde{L}_{2}\right\},$ then
		$\frac{1}{\tilde{L}_{3}}\leq \dfrac{\phi(x)}{\psi(x)}\leq \tilde{L}_{3}$
		for $x\in[h_{0}-\sigma_{1},h_{0}].$ Since $\phi^{\prime}(-h_{0})>0$ and $\psi^{\prime}(-h_{0})>0$, then there exists small $\sigma_{2}>0$ such that  $\phi^{\prime}(x)>\frac{\phi^{\prime}(-h_{0})}{2}>0$ and $\psi^{\prime}(x)>\frac{\psi^{\prime}(-h_{0})}{2}>0$ for any $x\in [-h_{0},-h_{0}+\sigma_{2}]$.
		Thus, there exists $\tilde{L}_{4}$ such that
		$\frac{1}{\tilde{L}_{4}}\leq \dfrac{\phi(x)}{\psi(x)}\leq \tilde{L}_{4}$
		for $x\in [-h_{0},-h_{0}+\sigma_{2}]$. Since $\phi(x)>0, \psi(x)>0$, then there exists $\tilde{L}_{5}>0$ such that $\frac{1}{\tilde{L}_{5}}\leq \dfrac{\phi(x)}{\psi(x)}\leq \tilde{L}_{5}$ for $x\in [-h_{0}+\sigma_{2},h_{0}-\sigma_{1}]$.
		Therefore, let $\tilde{L}=\max\limits\{\tilde{L}_{3},\tilde{L}_{4},\tilde{L}_{5}\}$, then (\ref{4.19}) holds.
		
		Let
		\begin{equation}\label{4.20}
		\begin{aligned}
		&\vartheta(t)=h_{0}(1+\delta-\frac{\delta}{2}e^{-\delta t}),
		\\&\overline{U}(x,t)=a_0 e^{-\delta t}\phi\left(\frac{x h_{0}}{\vartheta(t)}\right)e^{{\frac{\mu}{2D_{1}}}\left(1-\frac{h_{0}}{\vartheta(t)}\right)x},
		\\&\overline{V}(x,t)=a_0 e^{-\delta t}\psi\left(\frac{x h_{0}}{\vartheta(t)}\right),
		\end{aligned}
		\end{equation}
		for any $x\in(-\vartheta(t),\vartheta(t))$ and $t\geq0,$ where $a_0>0$ and $0<\delta\ll 1$ such that
		\begin{equation}\label{4.21}
		\begin{aligned}
		&-\delta-\frac{L h^2_{0}}{\vartheta^2(t)} \frac{\delta^2}{2}-\dfrac{\mu h^2_{0}}{4D_{1}}\dfrac{\delta^2}{\vartheta(t)}+\frac{\mu^2}{4D_{1}}\left(1-\dfrac{h^2_{0}}{\vartheta^2(t)}\right)+\gamma\left(1-\dfrac{h^2_{0}}{\vartheta^2(t)}\right)\\&+\lambda_{0}\dfrac{h^2_{0}}{\vartheta^2(t)}
		+a_{1}N_{1}\tilde{L}\left(\dfrac{h^2_{0}}{\vartheta^2(t)}-\frac{\delta}{1+\frac{\delta}{2}}\right)\geq 0,
		\end{aligned}
		\end{equation}
		and
		\begin{equation}\label{4.22}
		\begin{aligned}
		&-\delta -\dfrac{L h^2_{0}}{\vartheta^2(t)} \frac{\delta^2}{2} +a_{2}N_{2}\tilde{L}\left(\dfrac{h^2_{0}}{\vartheta^2(t)}-1\right)+d\left(1-\dfrac{h^2_{0}}{\vartheta^2(t)}\right)+\lambda_{0}\dfrac{h^2_{0}}{\vartheta^2(t)}\geq 0.
		\end{aligned}
		\end{equation}
		Further, direct calculation gives
		\begin{equation}\label{4.23}
		\begin{aligned}
		&\overline{U}_{t}-D_{1}\overline{U}_{xx}+\mu\overline{U}_{x}-a_{1}(N_{1}-\overline{U})\overline{V}+\gamma\overline{U}
		\\&\geq \overline{U}_{t}-D_{1}\overline{U}_{xx}+\mu\overline{U}_{x}
		-a_{1}N_{1}\overline{V}+\gamma\overline{U}
		\\&=-\delta \overline{U}-\frac{xh^2_{0}}{\vartheta^2(t)} \frac{\delta^2}{2} \dfrac{\phi^{\prime}}{\phi}\overline{U}+\dfrac{\mu h^2_{0}x}{4D_{1}}\dfrac{\delta^2}{\vartheta^2(t)}\overline{U}+\frac{\mu^2}{4D_{1}}\left(1-\dfrac{h^2_{0}}{\vartheta^2(t)}\right)\overline{ U}
		\\&+a_{1}N_{1}\overline{V}\left(\dfrac{h^2_{0}}{\vartheta^2(t)}e^{{\frac{\mu}{2D_{1}}}\left(1-\frac{h_{0}}{\vartheta(t)}\right)x}-1\right)+\gamma\left(1-\dfrac{h^2_{0}}{\vartheta^2(t)}\right)\overline{V}+\lambda_{0}\dfrac{h^2_{0}}{\vartheta^2(t)}\overline{U}
		\\&\geq\overline{U}\left(-\delta-\frac{L h^2_{0}}{\vartheta^2(t)} \frac{\delta^2}{2}-\dfrac{\mu h^2                                                                                                                                                        _{0}}{4D_{1}}\dfrac{\delta^2}{\vartheta(t)}+\frac{\mu^2}{4D_{1}}\left(1-\dfrac{h^2_{0}}{\vartheta^2(t)}\right)\right)
		\\&+\overline{U}\left(\gamma\left(1-\dfrac{h^2_{0}}{\vartheta^2(t)}\right)+\lambda_{0}\dfrac{h^2_{0}}{\vartheta^2(t)}+a_{1}N_{1}\tilde{L}\left(\dfrac{h^2_{0}}{\vartheta^2(t)}-\frac{\delta}{1+\frac{\delta}{2}}\right)\right)
		\\&\geq 0
		\end{aligned}
		\end{equation}
		and
		\begin{equation}\label{4.24}
		\begin{aligned}
		&\overline{V}_{t}-D_{2}\overline{V}_{xx}-a_{2}(N_{2}-\overline{V})\overline{U}+d\overline{V}
		\\&\geq\overline{V}_{t}-D_{2}\overline{V}_{xx}-a_{2}N_{2}\overline{U}+d\overline{V}
		\\&=-\delta \overline{V}-\dfrac{x h^2_{0}}{\vartheta^2(t)} \frac{\delta^2}{2}  \dfrac{\psi^{\prime}}{\psi}\overline{V}+a_{2}N_{2}\dfrac{\phi}{\psi}\overline{V} \left(\dfrac{h^2_{0}}{\vartheta^2(t)}-1\right)
		\\&+d\overline{V}\left(1-\dfrac{h^2_{0}}{\vartheta^2(t)}\right)+\lambda_{0}\overline{V}\dfrac{h^2_{0}}{\vartheta^2(t)}
		\\&\geq\overline{V}\left(-\delta -\dfrac{L h^2_{0}}{\vartheta^2(t)} \frac{\delta^2}{2} +a_{2}N_{2}\tilde{L}\left(\dfrac{h^2_{0}}{\vartheta^2(t)}-1\right)+d\left(1-\dfrac{h^2_{0}}{\vartheta^2(t)}\right)+\lambda_{0}\dfrac{h^2_{0}}{\vartheta^2(t)}\right)
		\\&\geq0
		\end{aligned}
		\end{equation}
		for any $x\in(-\vartheta(t),\vartheta(t))$ and $t\geq0.$
		
		Take $a_0=\frac{\delta^2h_{0}}{2\nu e^{{\frac{\mu}{2D_{1}}}h_{0}\delta}}\min\limits\left\{\frac{-1}{\phi^{\prime}(h_{0})},\frac{1}{\phi^{\prime}(-h_{0})} \right\}$, then
		
		\begin{equation}\label{4.25}
		\left\{\begin{array}{ll}{\overline{U}_{t}-D_{1} \overline{U}_{x x} \geq -\mu\overline{U}_{x}+a_{1}\left(N_{1}-\overline{U}\right) \overline{V}-\gamma \overline{U},} & {-\vartheta(t)<x<\vartheta(t), \quad t\geq0,} \\ {\overline{V}_{t}-D_{2} \overline{V}_{x x} \geq a_{2}\left(N_{2}-\overline{V}\right) \overline{U}-d \overline{V},} & {-\vartheta(t)<x<\vartheta(t), \enspace t\geq0,} \\ {\overline{U}(0, t) \geq U(0, t), \quad \overline{V}(0, t) \geq V(0, t),} & {t\geq0,} \\ {\overline{U}(x, t)\geq 0,\overline{V}(x, t)\geq 0,} & {x= \pm \vartheta(t), \quad t\geq0,} \\ {\vartheta^{\prime}(t) \geq-\nu \overline{U}_{x}(\vartheta(t), t),-\vartheta^{\prime}(t) \leq-\nu \overline{U}_{x}(-\vartheta(t), t),} & {\enspace t\geq 0.} \end{array}\right.
		\end{equation}
		If
		\begin{equation}\label{4.26}
		\begin{aligned}
		&||U_{0}||_{L_{\infty}}\leq a_0 \min\limits_{x\in [-h_{0},h_{0}]}\phi(\frac{x}{1+\delta})e^{{\frac{\mu}{2D_{1}}}(\frac{\delta}{1+\delta})x},
		\\&||V_{0}||_{L_{\infty}}\leq a_0 \min\limits_{x\in [-h_{0},h_{0}]}\psi(\frac{x}{1+\delta}),
		\end{aligned}
		\end{equation}
		then
		\begin{equation}\label{4.27}
		\begin{aligned}
		&\overline{U}(x, 0)=a_0\phi(\frac{x}{1+\delta})e^{{\frac{\mu}{2D_{1}}}(\frac{\delta}{1+\delta})x}
		\geq {U_{0}(x)},
		\\&\enspace\overline{V}(x, 0)=a_0\psi(\frac{x}{1+\delta})\geq V_{0}(x).
		\end{aligned}
		\end{equation}
		Therefore, $(\overline{ U}(x,t),\overline{V}(x,t);-\vartheta(t),\vartheta(t))$ is an upper solution of (\ref{1.3}). Hence, $h(t)\leq \vartheta(t), g(t)\geq -\vartheta(t)$, then $h_{\infty}-g_{\infty}\leq \lim\limits_{t\rightarrow +\infty}2\vartheta(t)\leq 2h_{0}(1+\delta).$   By Theorem \ref{t42}, we can get $\lim\limits_{t\rightarrow +\infty}||U(\cdot,t)||_{C(g(t),h(t))}=\lim\limits_{t\rightarrow +\infty}||V(\cdot,t)||_{C(g(t),h(t))}=0.$
	\end{proof}
	
	\begin{remark}\label{c46}\quad
		When given initial data $U_0(x)$ and $V_0(x)$ are small enough, if $\mu=0,$ then $R^D_{0}((-h_{0},h_{0}),$ $0,                                                                                                                                                                                                                                                                                                                                                                                                                                                            D_{1},D_{2})=1,$ the disease is spreading according to the arrguments in \cite{lin2017spatial}. However, if $\mu\neq 0,$ then
		$R^D_{0}((-h_{0},h_{0}),\mu,D_{1},D_{2})<1,$ the disease is vanishing by Theorem \ref{t45}.
	\end{remark}
	
	In fact,  even if $U_0(x)$ and $V_0(x)$ are not small, the disease will be extinct when the expanding capability $\nu$ is sufficiently small. Detailed proof can refer to Lemma 3.8 in \cite{du2010spreading}.
	\begin{theorem}\label{t47}\quad
		If $R^F_0(0)=R^D_{0}((-h_{0},h_{0}),\mu,                                                                                                                                                                                                                                                                                                                                                                                                                                                            D_{1},D_{2})<1,$ then there exists a small $\nu^*>0$ depending on $U_0$ and $V_0$ such that
		$h_\infty-g_\infty <\infty$ and
		$\lim\limits_{t\rightarrow +\infty}||U(\cdot,t)||_{C(g(t),h(t))}=\lim\limits_{t\rightarrow +\infty}||V(\cdot,t)||_{C(g(t),h(t))}=0$ when $\nu<\nu^*$.
	\end{theorem}
	\section{\textbf{The spreading regime of WNv}}\label{s5}
	\noindent
	Next, we will discuss the spreading conditions of the disease and investigate the impact of $R^F_0(t)$ on the infected habitats and the densities of mosquitoes. It implies that the spreading will occur when $R^F_0(0)\geq1$.
	\begin{theorem}\label{t51}\quad
		If $R^F_{0}({0})=R^D_0((-h_0,h_0),\mu,D_1,D_2)\geq 1$, then $ h_{\infty}-g_{\infty}=\infty$ and $$\lim\limits_{t\rightarrow +\infty}inf||U(\cdot,t)||_{C(g(t),h(t))}>0, \lim\limits_{t\rightarrow +\infty}inf||V(\cdot,t)||_{C(g(t),h(t))}>0.$$ It means that the disease will spread.
	\end{theorem}
	\begin{proof}
		In the case of $R^F_{0}({0})=R^D_0((-h_0,h_0),\mu,D_1,D_2)>1$, by Lemma \ref{l31}, there exists a principal eigenvalue $\lambda_0<0$ with positive eigenvalue function $(\phi(x),\psi(x))$ for the following problem
		\begin{equation}\label{5.01}
		\left\{\begin{array}{ll}{-D_{1} \phi_{x x}=-\mu\phi_{x}+a_{1} N_{1} \psi-\gamma \phi+\lambda_0\phi,} & {x\in (-h_0,h_0),} \\ {-D_{2} \psi_{x x}=a_{2} N_{2} \phi-d \psi+\lambda_0\psi,} & {x\in (-h_0,h_0), } \\ {\phi(x)=\psi(x)=0,} & { x=\pm h_0 .} \end{array}\right.
		\end{equation}
		Let us construct a lower solution to system (\ref{1.3}). Set
		$$\underline U(x,t)=\upsilon \phi(x),	\underline V(x,t)=\upsilon \psi(x),$$ where $x\in[-h_0,h_0],t\geq 0$ and $0<\upsilon\ll1 $ .
		
		Simple computation gives
		\begin{equation}\label{5.02}
		\begin{aligned}
		&\underline{U}_{t}-D_{1} \underline{U} _{x x}+\mu \underline{U}_x-a_1(N_1 - {\underline U}){\underline {V}}+\gamma \underline U
		\\&=\upsilon (-D_1\phi_{xx}+\mu \phi_x-a_1N_1\psi+\upsilon a_1 \phi\psi+\gamma \phi)
		\\&= \upsilon(\lambda_{0}\phi+\upsilon a_1 \phi\psi)
		\end{aligned}
		\end{equation}
		and
		\begin{equation}\label{5.03}
		\begin{aligned}
		&\underline{V}_{t}-D_{2} \underline{V}_{x x}-a_2(N_2-\underline {V}) \underline {U}+d \underline V
		\\&=\upsilon (-D_2\psi_{xx}-a_2N_2\phi+\upsilon a_2 \phi\psi+d \phi)
		\\&= \upsilon(\lambda_{0}\psi+\upsilon a_2 \phi\psi).
		\end{aligned}
		\end{equation}
		Since $\lambda_{0}<0$, recalling that $\phi^\prime(-h_0), \psi^\prime(-h_0)>0$ and $\phi^\prime(h_0), \psi^\prime(h_0)<0$, take $\upsilon$ sufficiently small such that
		\begin{equation}\label{5.04}
		\left\{\begin{array}{ll}{\underline{U}_{t}-D_{1} \underline{U}_{x x} \leq -\mu\underline{U}_{x}+a_{1}\left(N_{1}-\underline{U}\right) \underline{V}-\gamma \underline{U},} & {-h_0<x<h_0, \quad t>0,} \\ {\underline{V}_{t}-D_{2} \underline{V}_{x x} \leq a_{2}\left(N_{2}-\underline{V}\right) \underline{U}-d \underline{V},} & {-h_0<x<h_0, \enspace t>0,}  \\ {\underline{U}(x, t)\leq 0,\underline{V}(x, t)= 0,} & {x= \pm h_0, \quad t>0,} \\{0=(-h_0)^\prime\geq -\nu \underline U_x(-h_0,t),} & {t>0,}\\{0=(h_0)^\prime\leq -\nu \underline U_x(h_0,t),} & {t>0,}\\ {\underline{U}(0, t) \leq U(0, t), \quad \underline{V}(0, t) \leq V(0, t),} & {t>0.} \end{array}\right.
		\end{equation}
		Thus, by Lemma \ref{l22}, $U(x,t)\geq\underline U(x,t)$ and $V(x,t)\geq \underline V(x,t)$ for $x\in[-h_0,h_0]$ and $t\geq 0.$
		It implies that $\lim\limits_{t\rightarrow +\infty}inf||U(\cdot,t)||_{C(g(t),h(t))}\geq\upsilon \phi(0)>0$ and $\lim\limits_{t\rightarrow +\infty}inf||V(\cdot,t)||_{C(g(t),h(t))}\geq\upsilon \psi(0)>0.$ Moreover, by Theorem \ref{t42}, we get $h_\infty-g_\infty=\infty.$
		
		For $R^F_0(0)=R^D_0((-h_0,h_0),\mu,D_1,D_2)=1$,  by Property \ref{p32}, then $R^F_0(t_0)>R^F_0(0)=1$ for any $t_0>0$. Take the initial time from 0 to $t_0(>0)$ and repeat the above procedures, we can get $h_\infty-g_\infty=\infty.$
	\end{proof}
	\begin{remark}\label{r51}
		Assume that $(H)$ holds. $R^F_{0}(t_{0})\geq 1$ for some $t_{0}\geq 0$ if and only if the disease will spread. Indeed,  if $R^F_0(t)<0$ for any $t\geq 0$, then $h_\infty-g_\infty<\infty$, that is, the disease will be extinct.
	\end{remark}
	
	The following theorem can be obtained by constructing upper and lower sulution and analysis about the steady state of (\ref{1.3}) by Poincar{\'{e}}-Bendixson theorem and Proposition 2.1 in \cite{lewis2006traveling}.
	\begin{theorem}	\label{t52}
		When spreading occurs, the problem (\ref{1.3}) admits a unique coexistent steady state $E^{*}=(U^{*},V^{*})$, where $(U^{*},V^{*})$ is the unique globally asymptoyic stable endemic equilibrium of the following equation
		\begin{equation}\label{5.1}
		\left\{\begin{array}{ll}{\frac{du}{dt}=a_{1}(N_{1}-u)v-\gamma u,} & {t>0,} \\{\frac{dv}{dt}=a_{2}(N_{2}-v)u-d v,} & {t>0,} \end{array}\right.
		\end{equation}
		where $$U^{*}=\dfrac{a_1 a_2 N_1 N_2-\gamma d}{a_1 a_2 N_2 +a_2 \gamma },\enspace V^{*}=\dfrac{a_1 a_2 N_1 N_2-\gamma d}{a_1 a_2 N_1 + a_1 d}.$$
	\end{theorem}
	
	The following theorem implies the disease will spread when the expanding capability is sufficiently large, which is silimar to the arguments of \cite{tarboush2017spreading}.
	\begin{theorem}\label{t53}\quad
		Assume that $(H)$ holds. If $R^F_0(0)<1,$ then $h_\infty-g_\infty=\infty$ when $\nu$ is large enough.
	\end{theorem}
	
	In view of the above arguments, we can give the vanishing-spreading dichotomy regines of WNv.
	\begin{theorem}\label{t54}\quad
		Assume that $(H)$ holds. Let $(U(x,t),V(x,t);g(t),h(t))$ be the solution of system (\ref{1.3}), then the vanishing-spreading dichototmy regines hold:
		\\(1)Vanishing: $h_\infty-g_\infty<\infty$ and $\lim\limits_{t\rightarrow +\infty}(||U(\cdot,t)||_{C(g(t),h(t))}+||V(\cdot,t)||_{C(g(t),h(t))})=0$;
		\\(2)Spreading: $h_\infty-g_\infty=\infty$ and $\lim\limits_{t\rightarrow +\infty}U(x,t)=U^*,
		\lim\limits_{t\rightarrow +\infty}V(x,t)=V^*$ uniformly for $x$ in any compact subset of $\mathbb{R}.$
	\end{theorem}
	\section{Asymptotic spreading speed}\label{s6}
	In order to prevent the WNv from dispersing, it is essential to investigate the asymptotic spreading speed of the infected boundary.
	In Section 7 of \cite{lin2017spatial}, Lin and Zhu compared the definition of the minnimal wave speed, the spreading speed with the asymptotic spreading speed. Now, we aim to give the estimates of the asymptotic spreading speeds about the leftward front and rightward front for system (\ref{1.3}). For this purpose, we first recall a lemma.
	\begin{lemma}[Theorem 3.2 \cite{wang2019spreading}]\label{l55}\quad
		Assume that $a_{1}a_{2}N_{1}N_{2}>\gamma d$, then there exists a constant $c^{*}>0$ such that for every $c\in [0,c^{*}),$ the system
		\begin{equation}\label{5.2}
		\left\{\begin{array}{ll}
		{D_{1}u^{\prime\prime}-cu^{\prime}+a_{1}(N_{1}-u)v-\gamma u=0,} & { 0<s<\infty,} \\{D_{2}v^{\prime\prime}-cv^{\prime}+a_{2}(N_{2}-v)u-d v=0,} & {0<s<\infty,} \\ {(u(0),v(0) =(0,0),(u(\infty),v(\infty)) =(U^{*},V^{*})}
		\end{array}\right.
		\end{equation}
		admits a strictly increasing solution $(u_c,u_c)\in C^2(\mathbb{R^+})\times C^2(\mathbb{R^+})$ and there exists unique $c_\nu\in (0,c^*)$ such that $\nu u_{c_\nu}^{\prime}(0)=c_\nu$ for any $\nu>0.$	
	\end{lemma}
	
	When the spreading happens, considering the small advection, we will give a sharper estimate for different asymptotic spreading speeds of the leftward and rightward fronts of (\ref{1.3}). This is a main contribution of this paper in studying WNv.
	\begin{theorem}\label{t56}\quad
		Assume that $0<\mu<\mu^{*}$ and $a_1a_2N_1N_2>d\gamma$. Let $(U,V;g,h)$ be a solution of (\ref{1.3}) with $h_{\infty}-g_{\infty}=\infty,$                                                                                     the asymptotic spreading speeds of the leftward front and rightward front satisfy
		\begin{equation}\label{5.3}
		\lim\limits_{t\rightarrow +\infty}\sup\limits\frac{-g(t)}{t}\leq c_{\nu}\leq \lim\limits_{t\rightarrow +\infty}\inf\limits\frac{h(t)}{t}.
		\end{equation}
		where $c_{\nu}$ is the asymptotic spreading speed of the problem (\ref{1.3}) without the advection term.
	\end{theorem}
	\begin{proof}\quad
		We will divide the proof of this theorem into two steps.
		
		\textbf{Step 1.} we will show
		\begin{equation}\label{5.4}
		\lim\limits_{t\rightarrow +\infty}\inf\limits\frac{h(t)}{t}\geq c_{\nu}.
		\end{equation}
		For the small $\omega >0,$ construct the following auxiliary system motivated by Wang et al. \cite{wang2017asymptotic}
		\begin{equation}\label{5.5}
		\left\{\begin{array}{ll}
		{D_{1}u_{\omega}^{\prime\prime}-c_{\omega}u_{\omega}^{\prime}+a_{1}(N_{1}-u_{\omega})v_{\omega}-(\gamma+2\omega) u_{\omega}=0,} & { 0<s<\infty,} \\{D_{2}v_{\omega}^{\prime\prime}-c_{\omega}v_{\omega}^{\prime}+a_{2}(N_{2}-v_{\omega})u_{\omega}-(d+2\omega) v_{\omega}=0,} & {0<s<\infty,} \\ {(u_{\omega}(0),v_{\omega}(0) =(0,0),(u_{\omega}(\infty),v_{\omega}(\infty)) =(u_{-2\omega}^{*},v_{-2\omega}^{*}),}
		\end{array}\right.
		\end{equation}
		where $$u_{-2\omega}^{*}=\dfrac{a_1 a_2 N_1 N_2-(\gamma+2\omega)(d+2\omega)}{a_1 a_2 N_2 +(\gamma+2\omega)a_2}, v_{-2\omega}^{*}=\dfrac{a_1 a_2 N_1 N_2-(\gamma+2\omega)(d+2\omega)}{a_1 a_2 N_1 +(d+2\omega)a_1}.$$
		According to Lemma \ref{l55}, there exists unique $c_{\omega}=c(\nu,\omega)>0$ such that (\ref{5.5}) has a unique strictly increasing solution $(u_\omega,v_\omega)$ satisfying
		$$\nu u_{\omega}^{\prime}(0)=c_\omega, \lim\limits_{\omega\rightarrow 0^{+}}c_\omega=c_\nu.$$
		Since $h_\infty-g_\infty=\infty$, then $$(U(x,t),V(x,t))\rightarrow (U^{*},V^*)=\left(\dfrac{a_1 a_2 N_1 N_2-\gamma d}{a_1 a_2 N_2 +\gamma a_2},\dfrac{a_1 a_2 N_1 N_2-\gamma d}{a_1 a_2 N_1 +d a_1}\right)\enspace as\enspace t\rightarrow \infty$$
		uniformly for $x$ in any compact subset of $\mathbb{R}$. Hence, there exist large $T_1>0$ and $L\in (0,h(T_1))$ such that $$ h(t)>L, \enspace(U(x,t),V(x,t))\geq (u_{-\omega}^{*},v_{-\omega}^{*})>(u_{-2\omega}^{*},v_{-2\omega}^{*})$$
		for $x\in[0,L]$ and $t\geq T_1,$
		where $$u_{-\omega}^{*}=\dfrac{a_1 a_2 N_1 N_2-(\gamma+\omega)(d+\omega)}{a_1 a_2 N_2 +(\gamma+\omega)a_2}, \enspace v_{-\omega}^{*}=\dfrac{a_1 a_2 N_1 N_2-(\gamma+\omega)(d+\omega)}{a_1 a_2 N_1 +(d+\omega)a_1}.$$
		Let
		\begin{equation}\label{5.6}
		\begin{aligned}
		&\underline{h}(t)=c_\omega(t-T_1)+L,t\geq T_1,
		\\&\underline{U}(x,t)=u_\omega(\underline{h}(t)-x),0\leq x\leq\underline{h}(t),t\geq T_1,
		\\&\underline{V}(x,t)=v_\omega(\underline{h}(t)-x),0\leq x\leq\underline{h}(t),t\geq T_1.
		\end{aligned}
		\end{equation}
		Then $\underline{h}(T_1)=L<h(T_1)$ and
		$(\underline{U}(x,T_1),\underline{V}(x,T_1))=(u_\omega(\underline{h}(T_1)-x),v_\omega(\underline{h}(T_1)-x))
		\leq (u_{-\omega}^{*},v_{-\omega}^{*})\leq ({U}(x,T_1),{V}(x,T_1))$
		for $x\in[0,\underline{h}(T_1)]$.
		And we can see
		$$(\underline{U}(x,t),\underline{V}(x,t))=(u_\omega(\underline{h}(t)-x),v_\omega(\underline{h}(t)-x))\leq (u_{-\omega}^{*},v_{-\omega}^{*})\leq ({U}(x,t),{V}(x,t))$$
		$$(\underline{U}(\underline{h}(t),t),\underline{V}(\underline{h}(t),t))=(0,0),\enspace \underline{h}^\prime(t)=c_\omega=-\nu \underline{U}_{x}(\underline{h}(t),t) $$
		for $t\geq T_1.$
		Moreover, since $u_\omega ^\prime>0, 0< \mu<\mu^{*},$ we can get
		\begin{equation}\label{5.7}
		\begin{aligned}
		&\underline{U}_{t}-D_{1} \underline{U}_{x x}=c_\omega u_\omega ^\prime-D_1 u_{\omega}^{\prime\prime}
		\\&=a_1 (N_1 -u_\omega) v_\omega -(\gamma+2\omega)u_\omega
		\\&\leq a_1 (N_1 -u_\omega) v_\omega -\gamma u_\omega+\mu u_{\omega} ^\prime
		\\&= a_{1}\left(N_{1}-\underline{U}\right) \underline{V}-\gamma \underline{U} -\mu\underline{U}_{x}
		\end{aligned}
		\end{equation}
		and
		\begin{equation}\label{5.8}
		\begin{aligned}
		&\underline{V}_{t}-D_{2} \underline{V}_{x x}=c_\omega v_\omega ^\prime-D_2 v_{\omega}^{\prime\prime}
		\\&=a_{2}(N_{2}-v_{\omega})u_{\omega}-(d+2\omega) v_{\omega}
		\\&\leq a_{2}(N_{2}-v_{\omega})u_{\omega}-d v_{\omega}
		\\&= a_{2}\left(N_{2}-\underline{V}\right) \underline{U}-d \underline{V}.
		\end{aligned}
		\end{equation}
		Therefore, in view of the Comprison Principle, we can get $h(t)\geq \underline{h}(t)\enspace for\enspace t\geq T_1$ and
		\begin{equation}\label{5.9}
		\lim\limits_{t\rightarrow +\infty}\inf\limits\frac{h(t)}{t}\geq c_{\omega}.
		\end{equation}
		It follows that
		$\lim\limits_{t\rightarrow +\infty}\inf\limits\frac{h(t)}{t}\geq c_{\nu}$
		as $\omega\rightarrow 0.$
		
		\textbf{Step 2.} we will show
		\begin{equation}\label{5.10}
		\lim\limits_{t\rightarrow +\infty}\sup\limits\frac{-g(t)}{t}\leq c_{\nu}.
		\end{equation}
		Considering the following ODE system
		\begin{equation}\label{5.11}
		\left\{\begin{array}{ll}{u^{\prime}=a_{1}(N_{1}-u)v-\gamma u,} & { t>0,} \\{v^{\prime}=a_{2}(N_{2}-v)u-d v,} & {t>0,} \\ {u(0)=\sup\limits_{x\in[-h_{0},h_{0}]}U_{0}(x),}\\{v(0)=\sup\limits_{x\in[-h_{0},h_{0}]}V_{0}(x),} \end{array}\right.
		\end{equation}
		by Comprison Principle, we can obtain
		$$U(x,t),V(x,t)\leq(u(t),v(t)), for \enspace x\in[g(t),h(t)], t>0.$$
		By Theorem \ref{t52}, $(u(t),v(t))\rightarrow (U^*,V^*)$ as $t\rightarrow \infty.$ Therefore, we get
		\begin{equation}\label{5.12}
		\lim\limits_{t\rightarrow +\infty}\sup\limits\max_{x\in[g(t),h(t)]}U(x,t)\leq U^*, \lim\limits_{t\rightarrow +\infty}\sup\limits\max_{x\in[g(t),h(t)]}V(x,t)\leq V^*.
		\end{equation}
		Moreover, for any given small $\sigma>0,$ there exists $T_2>0$ such that
		$$(U(x,t),V(x,t)\leq(	U^*,V^*)\leq (u^*_{\sigma},v^*_{\sigma}) $$
		for $t\geq T_2$ and $x\in [g(t),h(t)]$,
		where
		$$u_{\sigma}^{*}=\dfrac{a_1 a_2 N_1 N_2-(\gamma-\sigma)(d-\sigma)}{a_1 a_2 N_2 +(\gamma-\sigma)a_2}, v_{\sigma}^{*}=\dfrac{a_1 a_2 N_1 N_2-(\gamma-\sigma)(d-\sigma)}{a_1 a_2 N_1 +(d-\sigma)a_1}.$$
		For the fixed $\sigma>0,$ construct the following auxiliary system
		\begin{equation}\label{5.13}
		\left\{\begin{array}{ll}
		{D_{1}u_{\sigma}^{\prime\prime}-c_{\sigma}u_{\sigma}^{\prime}+a_{1}(N_{1}-u_{\sigma})v_{\sigma}-(\gamma-2\sigma) u_{\sigma}=0,} & { 0<s<\infty,} \\{D_{2}v_{\sigma}^{\prime\prime}-c_{\sigma}v_{\sigma}^{\prime}+a_{2}(N_{2}-v_{\sigma})u_{\sigma}-(d-2\sigma) v_{\sigma}=0,} & {0<s<\infty,} \\ {(u_{\sigma}(0),v_{\sigma}(0) =(0,0),(u_{\sigma}(\infty),v_{\sigma}(\infty)) =(u_{2\sigma}^{*},v_{2\sigma}^{*}),}
		\end{array}\right.
		\end{equation}
		where
		$$u_{2\sigma}^{*}=\dfrac{a_1 a_2 N_1 N_2-(\gamma-2\sigma)(d-2\sigma)}{a_1 a_2 N_2 +(\gamma-2\sigma)a_2}, v_{2\sigma}^{*}=\dfrac{a_1 a_2 N_1 N_2-(\gamma-2\sigma)(d-2\sigma)}{a_1 a_2 N_1 +(d-2\sigma)a_1}.$$
		According to Lemma \ref{l55}, there exists unique $c_{\sigma}=c(\nu,\sigma)>0$ such that (\ref{5.13}) has a unique strictly increasing solution $(u_\sigma,v_\sigma)$ satisfying
		$$\nu u_{\sigma}^{\prime}(0)=c_\sigma, \lim\limits_{\sigma\rightarrow 0^{+}}c_\sigma=c_\nu.$$
		Since $(u_\sigma(\infty),v_\sigma(\infty))=(u_{2\sigma}^{*},v_{2\sigma}^{*})$, then there exists a $S_0\gg1$ such that
		$$(u^*_{\sigma},v^*_{\sigma})<(u_\sigma(S_0),v_\sigma(S_0)).$$
		Let
		\begin{equation}\label{5.14}
		\begin{split}
		&\overline g(t)=-c_\sigma(t-T_2)-S_0+g(T_2),\enspace t\geq T_2
		\\&\overline{U}(x,t)=u_\sigma(x-\overline g(t)),\enspace \overline g(t)\leq x\leq 0,\enspace t\geq T_2,
		\\&\overline{V}(x,t)=v_\sigma(x-\overline g(t)),\enspace \overline g(t)\leq x\leq 0,\enspace t\geq T_2.
		\end{split}
		\end{equation}
		Therefore, $$\overline g(T_2)=-S_0+g(T_2)<g(T_2),\enspace
		\overline g^{\prime}(t)=-c_\sigma=-\nu u_\sigma^{\prime}(0)=-\nu\overline{U}_x(\overline g(t),t)$$ for $t\geq T_2,$
		and
		$$\overline{U}(x,T_2)=u_\sigma(x+S_0-g(T_2))\geq u_\sigma(S_0)\geq U(x,T_2),$$
		$$\overline{V}(x,T_2)=v_\sigma(x+S_0-g(T_2))\geq v_\sigma(S_0)\geq V(x,T_2) $$
		for  $x\in[\overline g(T_2), 0].$ Moreover,
		$$\overline U(\overline g(t),t)=u_\sigma(0)=0,\enspace \overline V(\overline g(t),t)=v_\sigma(0)=0,$$
		$$\overline U(0,t)=u_\sigma(-\overline g(t))=u_\sigma(c_\sigma(t-T_2)+S_0-g(T_2))\geq u_\sigma(S_0)\geq U(0,t),$$
		$$\overline V(0,t)=v_\sigma(-\overline g(t))=v_\sigma(c_\sigma(t-T_2)+S_0-g(T_2))\geq v_\sigma(S_0)\geq V(0,t)$$
		for $t\geq T_2.$
		Since $u_\sigma^{\prime}>0$ and $0<\mu<\mu^*,$ then
		\begin{equation}\label{5.15}
		\begin{aligned}
		&\overline{U}_{t}-D_{1} \overline{U}_{x x}=c_\sigma u_\sigma ^\prime-D_1 u_{\sigma}^{\prime\prime}
		\\&=a_1 (N_1 -u_\sigma) v_\sigma -(\gamma-2\sigma)u_\sigma
		\\&\geq a_1 (N_1 -u_\sigma) v_\sigma -\gamma u_\sigma-\mu u_{\sigma} ^\prime
		\\&= a_{1}\left(N_{1}-\overline{U}\right) \overline{V}-\gamma \overline{U} -\mu\overline{U}_{x}
		\end{aligned}
		\end{equation}
		and
		\begin{equation}\label{5.16}
		\begin{aligned}
		&\overline{V}_{t}-D_{2} \overline{V}_{x x}=c_\sigma v_\sigma ^\prime-D_2 v_{\sigma}^{\prime\prime}
		\\&=a_{2}(N_{2}-v_{\sigma})u_{\sigma}-(d-2\sigma) v_{\sigma}
		\\&\geq a_{2}(N_{2}-v_{\sigma})u_{\sigma}-d v_{\sigma}
		\\&= a_{2}\left(N_{2}-\overline{V}\right) \overline{U}-d \overline{V}.
		\end{aligned}
		\end{equation}
		Therefore, in view of Comparison Principle, we can get $g(t)\geq \overline{g}(t),\enspace t\geq T_2$ and
		\begin{equation}\label{5.17}
		\lim\limits_{t\rightarrow +\infty}\sup\limits\frac{-g(t)}{t}\leq c_{\sigma}.
		\end{equation}
		It follows that
		$\lim\limits_{t\rightarrow +\infty}\sup\limits\frac{-g(t)}{t}\leq c_{\nu}$
		as $\sigma\rightarrow 0.$
		Therefore, our proof is completed.
	\end{proof}
	\begin{remark}\label{r416}\quad
		When the disease is spreading, if the small advection rate $0<\mu<\mu^*$, the asymptotic spreading speed of the leftward front is less than the rightward front. Similarly, if $-\mu^*<\mu<0,$ the asymptotic spreading speed of the leftward front is more than the rightward front. This fact implies that the advection term plays an important role in influencing WNv progapation speed.
	\end{remark}
	
	\section{Numerical simulations}\label{s7}
	\noindent
	In this section, we will provide some numerical simulations of our results by applying the Newton-Raphson method and a similar method as Razvan and Gabriel\cite{stef2008numerical} for the free boundary problem to investigate the impact of advection term on the transmission of West Nile virus.
	
	Take some parameter values of (\ref{1.3}) from \cite{lewis2006traveling}:
	\begin{equation}\label{6.1}
	\begin{aligned}
	&\dfrac{N_2}{N_1}=20,\alpha_1=0.88,\alpha_2=0.16.
	\end{aligned}
	\end{equation}
	Let $D_1=6,D_2=1,h_0=15,\gamma=0.6,$
	\begin{equation}\label{6.2}
	\begin{aligned}
	U_0(x) =
	\begin{cases}
	
	0.1*cos(\frac{\pi x}{2h_0}),       & x\in[-h_0,h_0] \\
	
	0,  & x\notin[-h_0,h_0]
	\end{cases},
	V_0(x) =
	\begin{cases}
	
	2* cos(\frac{\pi x}{2h_0}),       & x\in[-h_0,h_0] \\
	
	0,  &  x\notin[-h_0,h_0]
	\end{cases}
	\end{aligned}
	\end{equation}
	
	\subsection{Advection affects the basic reproduction number}
	\noindent
	\\First, we will study the impact of advection on the basic reproduction number. Fix $$\nu=2, \beta=0.3, a_1=0.88\times 0.3,a_2=0.16\times 0.3,d=0.3.$$
	Take $\mu=0$ and $\mu=3$, respectively, then $$R_0:=\dfrac{a_1a_2N_1N_2}{\gamma d}=1.408>1, R^F_0(0,0):=R^D_{0}((-h_{0},h_{0}),0,                                                                                                                                                                                                                                                                                                                                                                                                                                  D_{1},D_{2})=1.2241>1,$$ $$R^F_0(0,3):=R^D_{0}((-h_{0},h_{0}),3,                                                                                                                                                                                                                                                                                                                                                                                                                                                         D_{1},D_{2})=0.7831<1.$$ We choose small $U_0(x)$ and $V_0(x)$ as (\ref{6.2}), it can be seen that the density of mosquitoes tends to a positive steady state from Fig.1~(a) with $R^F_0(0,0)>1$, that is, the disease will spread and the infected region will expand to the whole habitat; while the density of the infected mosquitoes diseases to 0 as $t\rightarrow \infty$ from Fig.1~(b) with $R^F_0(0,3)<0$, it implies that the disease will vanish and the infected habitat will limit to a bounded region. This comparison identifies the significient impact of advection term.
	\begin{figure}[t]
		\centering
		\begin{tabular}{cc}
			\begin{minipage}[t]{1.5in}
				\includegraphics[width=1.5\textwidth]{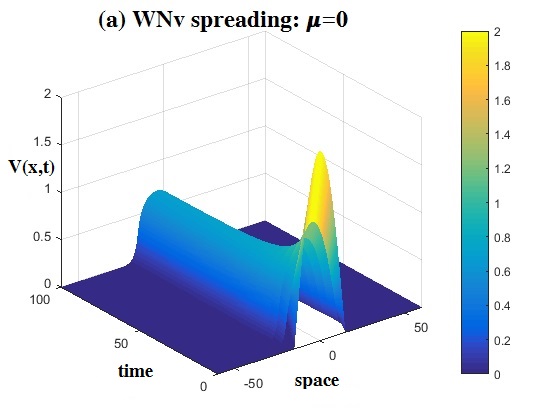}
				\label{fig:visual_smap_o}
			\end{minipage}
			\begin{minipage}[t]{1.5in}
				\includegraphics[width=1.5\textwidth]{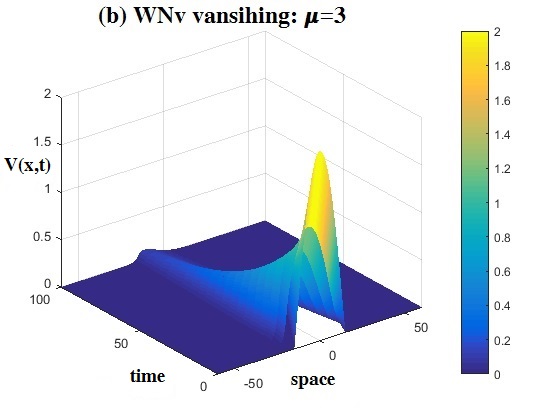}
			\end{minipage}
		\end{tabular}
		\caption{$R_0>R^F_0(0,0)>1>R^F_0(0,3)=0.7831$  }
		\label{fig:compare1}
	\end{figure}
	\subsection{Advection rate affects the spreading of boundary}
	\noindent
	\\Next, we will choose different advection intensities to study how they affect the spreading speed of boundaries. Fix
	$$\nu=4, \beta=0.5, a_1=0.88\times 0.5,a_2=0.16\times 0.5,d=0.029,$$ then $$\mu^{*}=2\sqrt{D_{1}(\dfrac{a_{1}a_{2}N_{1}N_{2}}{\gamma}-d)}\approx5.24.$$ Take $\mu=0$ and $\mu=2$, respectively, it is easy to see that the boundaries $x=g(t)$ and $x=h(t)$ expand differently with respect to $\mu$: the spreading speed of the right boundary is faster than the left boundary from Fig.2~(c) and Fig.2~(d) when $0<\mu<\mu^*$. Moreover, when spreading happens, the densities of birds and mosquitoes tend to a steady state in the long run and the infected habitat will expand to the whole area.
	
	\section{Discussions}\label{s8}
	\noindent
	The invasions of mosquitoes with West Nile virus, dengue fever virus and Zika virus or other virus have lead to many epidemic diseases risking people's health. Understanding the spatial dispersal  and dynamics of these virus plays a significant part in preventing and controlling infectious epidemics. In this paper, the dynamical behavior of a reaction-advection-diffusion WNv model with moving boundary conditions $x = g(t)$ and $x = h(t)$ about birds and mosquitoes is investigated by system (\ref{1.3}), the description of which is more compatible with the biological reality.
	
	Firstly, we introduce the spatial-temporal risk index $R^F_0(t)$ dependent on advection rate $\mu $ and spreading region $(g(t),h(t)$ as a threshold value to determine whether the disease will spread.   According to its definition, the advection rate can influence the values of $R^F_0(t)$ significantly (see to Section \ref{s3}). Next, we obtain some asymptotic properties of the temporal-spatial spreading of West Nile virus. Moreover, we get the sufficient and necessary conditions for spreading or vanishing of WNv. If
	$R^F_0(0)<1$ and initial data $||U_{0}||_{L_{\infty}},||V_{0}||_{L_{\infty}}$ are sufficiently small or $R^F_0(t)<1$ for any $t\geq0$, the epidemic will vanish and $\lim\limits_{t\rightarrow +\infty}(||U(\cdot,t)||_{C(g(t),h(t))}+||V(\cdot,t)||_{C(g(t),h(t))})=0$.  Under the assumption of $(H)$, $R^F_0(t_0)\geq 1$ for some $t_0\geq 0$ if and only if the disease will spread and the system (\ref{1.3}) admits a unique  stable positive equilibrium $(U^*,V^*)$ when the spreading occurs.
	
	\begin{figure}[t]
		\centering
		\begin{tabular}{cc}
			\begin{minipage}[t]{1.5in}
				\includegraphics[width=1.3\textwidth]{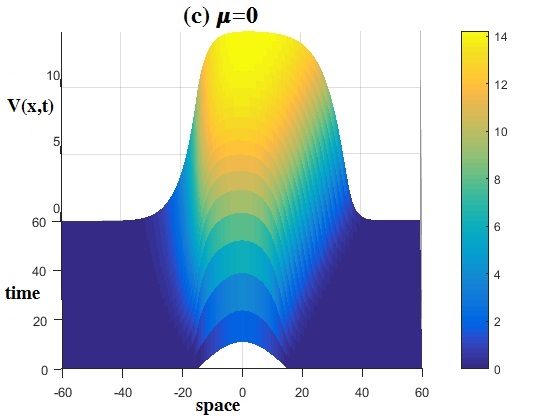}
				\label{fig:visual_smap_o}
			\end{minipage}
			\begin{minipage}[t]{1.5in}
				\includegraphics[width=1.3\textwidth]{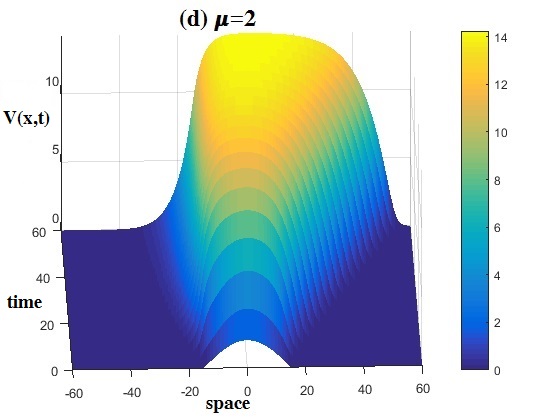}
			\end{minipage}
		\end{tabular}
		\caption{$\nu=4, a_1=0.88\times 0.5, a_2=0.16\times 0.5$, $\mu^{*}\approx5.24$
		}
		\label{fig:compare1}
	\end{figure}
	
	Assume that the habitat at far distance is in high risk, we mainly consider the effect of the small advection movement on the double spreading fronts of WNv. On the one hand, when the advection term $\mu$ becomes larger from $0$ to $\mu^{*}$, then disease becomes more difficult to spread. On the other hand, when the spreading occurs and $0< \mu<\mu^{*},$ we prove that the asymptotic spreading speed of the leftward is less than the rightward front:$	\lim\limits_{t\rightarrow +\infty}\sup\limits\frac{-g(t)}{t}\leq c_{\nu}\leq \lim\limits_{t\rightarrow +\infty}\inf\limits\frac{h(t)}{t}.$
	At last, we give some numerical simulations to investigate the impact of advection movement on the basic reproduction number and  the spreading of boundary (see to Section \ref{s7}). These simulation results are coincidate with the arguments by Maidana and Yang in \cite{maidana2009spatial}.
	
	Therefore, people can take measures to prevent diseases from dispersing according to our analysis and simulations. Furthermore, our reaction-advection-diffusion model of West Nile virus with double free boundaries can also be applied to study other cooperative systems of mosquito-borne diseases.
	In the future, the hign-dimensional system with free boundaries of WNv model will be explored. Moreover, the temperature, humidity and migration of birds with seasonality will be taken into our consideration to study the spreading of WNv.


\end{document}